\documentclass[10pt]{amsart}
\usepackage{amsmath,amsfonts,amsthm,amssymb}
\usepackage{datetime}
\usepackage[shortlabels]{enumitem}
\usepackage{tikz}
\usepackage{soul}
\usepackage{hyperref}

\newcommand{\R}{\mathbb{R}}
\newcommand{\N}{\mathbb{N}}
\newcommand{\id}{\mathrm{id}}
\newcommand{\ri}{\operatorname{ri}}
\newcommand{\range}{\operatorname{range}}
\newcommand{\inter}{\operatorname{int}}
\newcommand{\Fix}{F}
\newcommand{\ext}{\operatorname{ext}}

\newcommand{\inner}[1]{\left\langle #1 \right\rangle}
\newcommand{\Hp}{\mathcal{H}^+_\infty(T)}
\newcommand{\Gp}{\mathcal{G}_\infty(T)}
\newcommand{\Hm}{\mathcal{H}^-_\infty(T)}
\newcommand{\reach}{\operatorname{reach}}
\newcommand{\mylambda}{t}
\newcommand{\mymu}{s}

\theoremstyle{plain}
\newtheorem{theorem}{Theorem}[section]
\newtheorem{lemma}[theorem]{Lemma}

\newtheorem{corollary}[theorem]{Corollary}
\newtheorem{proposition}[theorem]{Proposition}

\theoremstyle{remark}
\newtheorem{remark}[theorem]{Remark}

\theoremstyle{definition}
\newtheorem{example}[theorem]{Example}

\numberwithin{equation}{section}

\begin{document}
\title{Nonexpansive maps with surjective displacement}
\author[B. Lins]{Brian Lins}
\date{}
\address{Brian Lins, Hampden-Sydney College}
\email{blins@hsc.edu}

\begin{abstract}
We investigate necessary and sufficient conditions for a nonexpansive map $f$ on a Banach space $X$ to have surjective displacement, that is, for $f - \id$ to map onto $X$.  In particular, we give a computable necessary and sufficient condition when $X$ is a finite dimensional space with a polyhedral norm.  We give a similar computable necessary and sufficient condition for a fixed point of a polyhedral norm nonexpansive map to be unique. We also consider applications to nonlinear Perron-Frobenius theory and suggest some additional computable sufficient conditions for surjective displacement and uniqueness of fixed points.    
\end{abstract}

\maketitle
 
\section{Introduction}


Given a nonexpansive map on a metric space, it is not always easy to determine whether the map has any fixed points. 
In Banach spaces with the fixed point property, any nonexpansive map that leaves a closed, bounded, convex set invariant has a fixed point.  
However, in applications it can be difficult to determine whether a nonexpansive map 
has any invariant, bounded, closed, convex sets. 

For finite dimensional Banach spaces, necessary and sufficient conditions for nonexpansive maps to have a nonempty and bounded set of fixed points were given in \cite{LLN16}.  When $X$ is a finite dimensional normed space with either a smooth norm or a polyhedral norm, \cite[Section 4]{LLN16} describes a non-deterministic algorithm that eventually halts (almost certainly) if and only if the fixed point set is nonempty and bounded.  

Somewhat surprisingly, it may be easier to determine whether a nonexpansive map $f:X \rightarrow X$ has the stronger property that the map $f+u$ has a fixed point for every $u \in X$. This is equivalent to $f$ having surjective displacement, that is $f-\id$ being onto. The main result of this paper is a computable test involving a finite set of limits which determines whether or not $f$ has surjective displacement when $X$ is a polyhedral normed space. A similar test can also determine whether a fixed point of a nonexpansive map on a polyhedral normed space is unique.

A motivation for this investigation comes from nonlinear Perron-Frobenius theory which seeks conditions for the existence and also uniqueness of eigenvectors of order-preserving and homogeneous maps. Some general results in this area include \cite{Birkhoff57,Morishima64,Oshime83,Nussbaum88,GaGu04,AkGaHo20}.
Many of these results rely on the close connection between nonlinear Perron-Frobenius theory and the theory of nonexpansive maps, see e.g. \cite{KoPr82,LemmensNussbaum}. 
Our results are inspired by a hypergraph condition from \cite{AkGaHo20}
that determines whether or not an order-preserving, additively homogeneous function $T: \R^n \rightarrow \R^n$ has several properties, one of which is that the functions $T+u$ have an additive eigenvector for every $u \in \R^n$ (see also \cite{AkGaHo15} where the hypergraph condition was first introduced for Shapley operators associated with certain two-player games).  
Order-preserving and additively homogeneous maps on $\R^n$ are nonexpansive with respect to the variation seminorm, so the results in \cite{AkGaHo15} and \cite{AkGaHo20} can be interpreted as fixed point results for a special class of nonexpansive maps. 
We will show that the results in \cite{AkGaHo20} can be derived from our general test for surjective displacement of a polyhedral norm nonexpansive map.

The paper is organized as follows.  We begin by collecting some facts about nonexpansive maps, Banach spaces, and the metric compactification and horofunctions of a Banach space. We use these facts in Section \ref{sec:surjDisp} to identify necessary and sufficient conditions for a nonexpansive map to have surjective displacement. The main result of Section \ref{sec:surjDisp} is Theorem \ref{thm:abcd} which summarizes these conditions 
for finite and infinite dimensional Banach spaces. In Section \ref{sec:polyhedral} we prove our main result which is a computable necessary and sufficient condition for a polyhedral norm nonexpansive map to have surjective displacement.  We look at applications of our main result to nonlinear Perron-Frobenius theory in Section \ref{sec:PF} and in Section \ref{sec:unique} we give a computable necessary and sufficient condition for a fixed point of a polyhedral norm nonexpansive map to be unique.  We conclude with a look at other computable sufficient conditions for a nonexpansive map to have surjective displacement or a unique fixed point.



\section{Preliminaries}

For any metric space $(X,d)$, a map $f: X \rightarrow X$ is \emph{nonexpansive} if $d(f(x),f(y)) \le d(x,y)$ for all $x, y \in M$.  Here we focus on nonexpansive maps $f$ defined on a real Banach space. For a Banach space $X$, we use the following notation for the norm $\|\cdot\|$, dual space $X^*$, and dual norm $\|\cdot\|_*$. Let $B_R$ denote the closed ball with radius $R$ and center $0$ in $X$ and let $B^*_R$ denote the closed ball of radius $R$ around $0$ in $X^*$. The identity map on $X$ is $\id$, and $J$ denotes the set-valued \emph{duality map} $J:X \rightrightarrows X^*$ given by  
$$J(x) := \{x^* \in X^* : \|x^* \|_* = \|x\|, \inner{x, x^*} = \|x\| \, \|x^*\|_* \}.$$

For any function $f: X \rightarrow X$, we refer to $f - \id$ as the \emph{displacement map} associated with $f$.  We say that $f$ has \emph{surjective displacement} if $f - \id$ maps onto $X$ and $f$ has \emph{dense displacement} if the range of $f-\id$ is dense in $X$.  For any $\delta \ge 0$, we define $\Fix_\delta(f) = \{x \in X : \|f(x) - x\| \le \delta \}$. When $\delta > 0$, we refer to $\Fix_\delta(f)$ as an \emph{approximate fixed point set} of $f$. Of course, $\Fix_0(f)$ is the set of fixed points of $f$.

A convex subset $F$ of a convex set $C$ is called a \emph{face} of $C$ if $tx+ (1-t)y \in F$ with $x,y \in C$ and $0 < t <1$ implies that $x$ and $y$ are contained in $F$. If $F$ is a face of $C$ and $F \ne C$, then $F$ is a \emph{proper face}.  The intersection of a supporting hyperplane with $C$ is called an \emph{exposed face}. All exposed faces are proper faces, but there are simple examples where proper faces are not exposed faces (see \cite[Section 18]{Rockafellar}). For any convex set $C$, we denote the extreme points of $C$ by $\ext C$ and the relative interior of $C$ by $\ri C$.  For a proper face of $F$ of the unit ball in a real Banach space, the \emph{dual face} $F^*$ is the image of $\ri F$ under the duality mapping $J$.  

A norm $\| \cdot \|$ on $X$ is \emph{polyhedral} if the unit ball $B_1$ is a convex polytope, that is, if it has a finite number of extreme points. A real Banach space $X$ is \emph{strictly convex} if $\|x+y\| < 2$ whenever $x,y \in X$ with $x \ne y$ and $\|x\|=\|y\| = 1$. It is \emph{uniformly convex} if for every $0 < \epsilon \le 2$, there is a $\delta > 0$ such that for any $x, y \in X$ with $\|x\|=\|y\| = 1$ and $\|x-y\| \ge \epsilon$, $\|x+y\| \le 2 - \delta$. A real Banach space $X$ is \emph{smooth} if $X^*$ is strictly convex, and \emph{uniformly smooth} if $X^*$ is uniformly convex. Note that uniformly convex and uniformly smooth Banach spaces are always reflexive  \cite[Corollary 3.21 and Theorem 3.31]{Brezis}.  

A Banach space $X$ has the \emph{fixed point property} if for every closed, bounded, convex subset $C \subset X$ and every nonexpansive map $f:C \rightarrow C$, $f$ has a fixed point. Browder \cite{Browder65} and G\"{o}hde \cite{Gohde65} independently proved that uniformly convex Banach spaces have the fixed point property, while Kirk proved that reflexive Banach spaces with normal structure have the fixed point property \cite{Kirk65}. It turns out that all uniformly convex Banach spaces have normal structure \cite[Theorem 4.1]{BeKi67}, while there are examples of reflexive Banach spaces with normal structure that are not isomorphic to any uniformly convex Banach space \cite{BeKiSt68}. Therefore Kirk's result is more general than the results of Browder and G\"{o}hde. Since then, other conditions have been found that are sufficient for a Banach space to have the fixed point property. It is a famous open problem whether or not all reflexive Banach spaces have the fixed point property.  See \cite{GF-LF-MN06} and the references therein for some recent contributions and a good introduction to the topic.  

\subsection{The metric compactification and horofunctions.}

In any metric space $(X,d)$, we can choose an arbitrary base point $b$, and consider the functions
$$h_y(x) = d(x,y) - d(b,y).$$
Each of these functions is 1-Lipschitz and $h_y(b) = 0$.  Therefore $h_y(x)$ takes values in the closed interval $[-d(x,b),d(x,b)] \subset \R$ for each $x \in X$, so each function $h_y$ can be identified with an element of the product space
$$\prod_{x \in X} [-d(x,b),d(x,b)]$$
which is a compact topological space by Tychonoff's theorem. The topology on the product space is equivalent to the topology of pointwise convergence, and the closure of the set $\{h_y: y \in X\}$ in this topology is known as the \emph{metric compactification of} $(X,d)$.  The elements of the metric compactification are called \emph{metric functionals}. Metric functionals include both the functions $h_y$ described above, and also functions on the boundary of the metric compactification, which are known as \emph{horofunctions}.  Any horofunction $h$ is a pointwise limit
$$h(x) = \lim_{\beta} d(x,y_\beta) - d(b,y_\beta)$$
where $y_\beta$ is a net in $(X,d)$. For 1-Lipschitz functions, the topology of pointwise convergence is equivalent to the topology of uniform convergence on compact subsets of $X$ by the Arzel\`{a}-Ascoli theorem. If the metric space $(X,d)$ is proper, i.e., all closed balls in $X$ are compact, then the topology of uniform convergence on compact subsets of $X$ is metrizable. In that case, we can use sequences $y_k$ to define horofunctions rather than nets. This is true, in particular, if $X$ is a finite dimensional Banach space. In addition, $d(b,y_k) \rightarrow \infty$ for any sequence $y_k$ that defines a horofunction in a proper metric space. Note that horofunctions are always 1-Lipschitz.  In a Banach space, horofunctions are also convex. Since the base point $b \in X$ is arbitrary, it is standard practice to choose $b = 0$ when $X$ is a Banach space. 

In a finite dimensional real Banach space, Walsh has shown (implicitly in \cite[Theorem 1.1]{Walsh07}), 
that a subset of the horofunctions known as the Busemann points are given by 
\begin{equation} \label{eq:Cormac}
h(x) = \sup_{\nu \in F^*} \inner{u-x, \nu},
\end{equation}
where $F^*$ is a proper face of $B_1^*$ and $u \in X$ satisfies $\sup_{\nu \in F^*} \inner{u,\nu} = 0$. For some finite dimensional Banach spaces all horofunctions are Busemann points, but this is not true in general.  The following is \cite[Theorem 1.2]{Walsh07}.

\begin{theorem}[Walsh] \label{thm:Cormac} Let $X$ be a finite dimensional real Banach space.  All horofunctions on $X$ are Busemann points given by \eqref{eq:Cormac} if and only if the faces of the dual unit ball $B_1^*$ are closed in the Painlev\'{e}-Kuratowski topology. This is true, in particular, if $X$ has a polyhedral norm or if $X$ is smooth. 
\end{theorem}

Even when a horofunction on a finite dimensional normed space is not a Busemann point, we can still make the following observation. Note that this result was part of the proof of \cite[Theorem 3.4]{LLN16}, however we reproduce the proof here for the reader's convenience.
\begin{lemma} \label{lem:deepdirection}
Let $X$ be a finite dimensional real Banach space. For any horofunction $h$ on $X$, there is a $v \in X$ with $\|v\|=1$ such that 
$$h(x+\mylambda v) = h(x) - \mylambda $$
for all $x \in X$ and $\mylambda \in \R$.  
\end{lemma}

\begin{proof}
There is a sequence $x_k$ with $\|x_k\| \rightarrow \infty$ such that 
$$h(x) = \lim_{k \rightarrow \infty} \|x-x_k\| - \|x_k\|.$$
We may assume, by passing to a subsequence if necessary, that $\lim_{k \rightarrow \infty} x_k/\|x_k\| = v$. Then 
\begin{align*}
\left|\left| x + t\frac{x_k - x}{\|x_k - x\|} - x_k \right|\right| - \|x_k\| &= \|x - x_k\| \left( 1 - \frac{t}{\|x_k-x\|} \right) - \|x_k\| \\
&= \|x-x_k\| - \|x_k\|-t.
\end{align*}
In the limit as $k \rightarrow \infty$, this equation becomes $h(x+tv) = h(x) - t$. 
\color{black}
\end{proof} 

Unfortunately, Lemma \ref{lem:deepdirection} does not hold for infinite dimensional Banach spaces. Guti\'{e}rrez \cite{Gutierrez19} has given a complete description of the horofunctions on infinite dimensional Hilbert spaces and $\ell_p$ spaces with $1 \le p < \infty$.  In some Banach spaces (including Hilbert spaces), there are non-trivial horofunctions $h$ that are bounded below. Even when a horofunction is not bounded below, it is possible that there is no direction $v \in X$ such that $\lim_{\mylambda \rightarrow \infty} h(\mylambda v) = - \infty$ as the following example shows.  

\begin{example}
By \cite[Theorem 3.6]{Gutierrez19}, the function 
$$h(x) = \sum_{i \in \N} (|x_i - 1| - 1)$$
is a horofunction on the real Banach space $\ell_1(\N)$. At the same time, it can be shown that $\lim_{\mylambda \rightarrow \infty} h(\mylambda v) = +\infty$ for all nonzero $v \in \ell_1(\N)$ even though $\inf_{x \in \ell_1(\N)} h(x) = -\infty$.   
\end{example}

However, when $X$ is uniformly smooth, the following result is an immediate consequence of \cite[Lemma 5.3]{Gutierrez19}. 
\begin{lemma}[Guti\'{e}rrez] \label{lem:Gu}
Let $X$ be a uniformly smooth real Banach space. If $h$ is a horofunction on $X$ with $\inf_{x \in X} h(x) = - \infty$, then there is a nonzero linear functional $\nu \in B_1^*$ such that $h(x) = - \inner{ x, \nu}$ for all $x \in X$.  
\end{lemma}

In the sequel, we will use horofunctions to help understand nonexpansive maps. One important result we will use is the following special case of \cite[Theorem 1]{GaVi12}. Note that all Banach spaces are metrically star-shaped, as is any convex subset of a Banach space.  
\begin{theorem}[Gaubert-Vigeral] \label{thm:GaVi}
Let $T$ be a nonexpansive self-map of a complete metrically star-shaped metric space $(X,d)$. If $\rho(T) = \inf_{y \in X} d(y,T(y)) > 0$, then there is a horofunction $h$ on $X$ such that $h(T(x)) \le h(x) - \rho(T)$ for all $x \in X$. 
\end{theorem}

\section{Surjective displacement} \label{sec:surjDisp}

The following theorem characterizes when nonexpansive maps on a finite dimensional normed space have surjective displacement. The equivalence of conditions (2), (3), and (4) was already observed by Hochart in \cite[Corollary 4]{Hochart19}. 

\begin{theorem} \label{thm:char}
Let $X$ be a finite dimensional real Banach space and suppose $f: X \rightarrow X$ is nonexpansive. The following are equivalent:
\begin{enumerate}
\item There is no horofunction $h$ and constant $c \in \R$ such that $h(f(x)) \le h(x) + c$ for all $x \in X$.
\item The approximate fixed point sets $\Fix_\delta(f)$ are bounded for all $\delta > 0$.
\item $f+u$ has a nonempty and bounded set of fixed points for all $u\in X$.
\item $f-\id$ maps onto $X$. 
\end{enumerate}
\end{theorem}

\begin{proof}
(1) $\Rightarrow$ (2). Suppose $F_\delta(f)$ is not bounded for some $\delta > 0$. Then there is a sequence $x_k$ such that $\|x_k \| \rightarrow \infty$ while $\|f(x_k) - x_k \| \le \delta$ for all $k \ge 0$.  Replacing $x_k$ with a subsequence if necessary, we may assume that the horofunction $h(x) = \lim_{k \rightarrow \infty} \|x - x_k\| - \|x_k\|$ is defined on $X$. Then
\begin{align*}
h(f(x)) &= \lim_{k \rightarrow \infty} \|f(x) - x_k\| - \|x_k\| & \\
&\le \liminf_{k \rightarrow \infty} \|f(x) - f(x_k)\| + \|f(x_k) - x_k\| - \|x_k\| & (\text{triangle inequality})\\
&\le \liminf_{k \rightarrow \infty} \|x - x_k \| + \delta - \|x_k\| & (\text{by nonexpansiveness})\\
&= h(x) + \delta &
\end{align*}
for all $x \in X$.  This contradicts (1).

(2) $\Rightarrow$ (3). Choose any $u \in X$. Let $x_k$ be defined recursively by $x_0=0$ and $x_{k+1} = f(x_k)+u$.  Since $f+u$ is nonexpansive, 
\begin{align*}
\|f(x_k) - x_k \| &\le \|f(x_k) + u - x_k\| + \|u\| & (\text{triangle inequality})\\
&\le \|f(0) + u\| + \|u\| & (\text{by nonexpansiveness})
\end{align*} 
for all $k$. Let $\delta = \|f(0) + u \| + \|u\|$.  Since the set $\Fix_\delta(f)$ is bounded, it follows that $x_k$ is a bounded orbit of $f+u$, and therefore $f+u$ must have a fixed point \cite[Proposition 3.2.4]{LemmensNussbaum}. The fixed point set for $f+u$ must also be bounded because it is contained in $\Fix_\delta(f)$. 

(3) $\Rightarrow$ (4). This is obvious since $-u$ is in the range of $f-\id$ if and only if $f+u$ has a fixed point.  

(4) $\Rightarrow$ (1). Consider any horofunction $h$ on $X$. By Lemma \ref{lem:deepdirection}, there is a $v \in X$ such that $h(x+\mylambda v) = h(x)-\mylambda$ for all $\mylambda \in \R$ and $x \in X$.  Since $f-\id$ is onto, we can choose $x \in X$ such that $f(x) = x+\mylambda v$.  Then $h(f(x)) = h(x) - \mylambda$, which proves (1). 
\end{proof}

In the rest of this section, we will consider necessary and sufficient conditions on any (finite or infinite dimensional) Banach space to have surjective displacement. Before proving the main result of the section, we need the following minor results.   
\begin{lemma} \label{lem:monotone}
Suppose that $X$ is a real Banach space and $f: X \rightarrow X$ is nonexpansive. For any $x \in X$ with $\|x\|=1$, and any $x^* \in J(x)$, the real function $\mylambda \mapsto \inner{f(\mylambda x) - \mylambda x, x^* }$ is monotone decreasing.  Consequently, $\lim_{\mylambda \rightarrow \infty}  \inner{f(\mylambda x) - \mylambda x, x^* }$ is either $-\infty$ or it converges to a finite value.  
\end{lemma}
\begin{proof}
Consider $\mylambda_1 < \mylambda_2$ and observe that 
\begin{align*}
\inner{f(\mylambda_2 x) - f(\mylambda_1 x), x^*} &\le \|f(\mylambda_2 x) - f(\mylambda_1 x)\| & \text{ (since } \|x^*\|_* = 1)\\
 &\le \|\mylambda_2 x - \mylambda_1 x\| & \text{ (by nonexpansiveness)}\\
 &= \inner{\mylambda_2 x - \mylambda_1 x, x^*}. & \text{ (since } t_2 > t_1 \text{ and } x^* \in J(x))
\end{align*}
Rearranging terms, we have $\inner{f(\mylambda_2x) - \mylambda_2 x, x^*} \le \inner{f(\mylambda_1 x) - \mylambda_1 x, x^*}$, as claimed.
\end{proof}

The next lemma is a semidifferentiability property of Banach space norms that is not new, but we could not find a reference. Before stating the result, we review some terminology about convex functions.  Let $X$ be a real Banach space and suppose that $f: X \rightarrow \R$ is a convex function.  An $x^* \in X^*$ is called a \emph{subgradient} of $f$ at $x$ if $f(z) - f(x) \ge \inner{z-x,  x^*}$
for all $z \in X$. The set of all subgradients at $x$ is called the \emph{subdifferential} of $f$ at $x$. For any $x \in X$ with $\|x\|=1$, the subdifferential for the norm at $x$ is precisely $J(x)$. 

\begin{lemma} \label{lem:subgradient}
Let $X$ be a real Banach space and let $x, y \in X$ with $\|x\|=1$.  Then there exists $x^* \in J(x)$ such that 
\begin{equation} \label{eq:subgradient}
\lim_{\mylambda \rightarrow \infty} \|\mylambda x - y\| - \inner{\mylambda x - y, x^*} = 0.
\end{equation}
\end{lemma}

\begin{proof}
If $X$ is finite dimensional, then \cite[Theorem 23.4]{Rockafellar} says that  
$$\lim_{\epsilon \rightarrow 0^+} \frac{\|x-\epsilon y\|- \|x\|}{\epsilon} = \sup_{x^* \in J(x)} \inner{-y,x^*}$$
since $J(x)$ is the subdifferential of the norm at $x$. 
Furthermore $J(x)$ is compact, so there is an $x^* \in J(x)$ where the supremum is attained, and then
$$\lim_{\epsilon \rightarrow 0^+} \frac{\|x-\epsilon y\|- \|x\|}{\epsilon} = -\inner{y,x^*}.$$
By replacing $\epsilon$ with $\mylambda^{-1}$, we have \eqref{eq:subgradient}. The infinite dimensional version follows by first applying the result to the finite dimensional subspace spanned by $x$ and $y$, and then extending the linear functional $x^*$ to all of $X$ using the Hahn-Banach theorem.
\end{proof}

For any nonzero $x \in X$, let $W_x = \{w \in X : \sup_{x^* \in J(x)} \inner{w,x^*} < 0 \}$.  Note that each $W_x$ is a nonempty open convex cone in $X$.  The following result says that if certain limits are bounded below, then the range of the displacement map does not intersect a translate of $W_x$ for some $x$. 
\begin{lemma} \label{lem:avoidedCone}
Let $X$ be a real Banach space and let $f: X \rightarrow X$ be nonexpansive. Let $x \in X$ with $\|x \| = 1$.
If $\lim_{\mylambda \rightarrow \infty} \inner{f(\mylambda x) - \mylambda x, x^*} \ge c$
for all $x^* \in J(x)$, then $\range(f-\id)$ does not intersect the open affine cone $W_x + cx$.
\end{lemma}
\begin{proof}
By Lemma \ref{lem:subgradient}, for any $y \in X$ there is an $x^* \in J(x)$ such that 
\begin{align*}
0 &= \lim_{t \rightarrow \infty} \|t x - y \| - \inner{t x - y, x^*} \\
&\ge \limsup_{t \rightarrow \infty} \|f(t x) - f(y) \| - \inner{t x - y, x^*} & (\text{nonexpansiveness}) \\ 
&\ge \limsup_{t \rightarrow \infty} \inner{f(\mylambda x) - f(y), x^*} - \inner{t x - y, x^*} & (\text{since } \|x^*\|_* = 1) \\
&= \lim_{t \rightarrow \infty} \inner{f(tx)-tx,x^*} - \inner{f(y)-y,x^*} \\
&\ge c - \inner{f(y)-y,x^*}.
\end{align*}
This means that $\sup_{x^* \in J(x)} \inner{f(y)-y,x^*} \ge c$ so $f(y)-y \notin W_x + cx$.
\end{proof}

The next observation is also about the range of displacement maps, this time when $f$ is nonexpansive with respect to a uniformly smooth norm.  This lemma can be proved directly from a more general result \cite[Theorem 5.2]{Reich82} which says that the closure of the range of an $m$-accretive operator on a smooth Banach space is always convex. Here we take a different approach using horofunctions.

\begin{lemma} \label{lem:convexRange}
Let $X$ be a uniformly smooth real Banach space and suppose that $f:X \rightarrow X$ is nonexpansive.  The closure of the range of $f - \id$ is convex. 
\end{lemma}

\begin{proof}
Suppose that $u \in X \backslash \overline{\range}(f-\id)$.  Then $\inf_{x \in X} \|f(x)-u - x \| = \rho > 0$. By Theorem \ref{thm:GaVi}, there is a horofunction $h$ on $X$ such that $h(f(x)-u) \le h(x) - \rho$ for all $x \in X$.  This horofunction cannot be bounded below since the iterates $(f-u)^k(0)$ have $h((f-u)^k(0)) \rightarrow -\infty$. So Lemma \ref{lem:Gu} implies that there is a nonzero linear functional $\nu \in B_1^*$ such that $h(x) = -\inner{x,\nu}$ for all $x$. Then 
$$-\inner{f(x)-u, \nu} \le -\inner{x,\nu} - \rho$$
for all $x \in X$.  Equivalently,
$$\inner{f(x)-x, \nu} \ge \inner{u, \nu} + \rho.$$  
Therefore the linear functional $\nu$ separates $u$ from $\overline{\range}(f-\id)$.  This means that $\overline{\range}(f-\id)$ is the intersection of an infinite family of closed half-spaces (one for each $u \notin \overline{\range}(f-\id)$), and as such, it must be a convex set. 
\end{proof}



The last preliminary observation before the main result of this section is the following fixed point theorem \cite[Theorem 2.4]{KaKi06}.  

\begin{theorem}[Kaewcharoen-Kirk] \label{thm:Penot}
Let $X$ be a Banach space which has the fixed point property, let $C$ be a closed, convex subset of $X$, and suppose that $f: C \rightarrow C$ is nonexpansive.  If $\Fix_\delta(f) = \{x \in C : \|f(x) - x\| \le \delta \}$ is bounded and nonempty for some $\delta \ge 0$, then $f$ has a fixed point. 
\end{theorem}

%

The main result of this section is the following theorem.  Condition \ref{item:SaBounded} in the theorem is sufficient for a nonexpansive map to have surjective displacement when the Banach space has the fixed point property, while condition \ref{item:nuLim} is necessary for surjective displacement on any Banach space.

\begin{theorem} \label{thm:abcd}
Let $f: X \rightarrow X$ be a nonexpansive map on a real Banach space $X$.  Consider the following possible conditions on $f$.  
\begin{enumerate}[(a)]
\item \label{item:SaBounded} The approximate fixed point sets $\Fix_\delta(f)$ are bounded for all $\delta > 0$. 
\item \label{item:fu} $f+u$ has a nonempty and bounded set of fixed points for all $u\in X$.
\item \label{item:surjDisp} $f - \id$ maps onto $X$.  
\item \label{item:denseDisp} The range of $f - \id$ is dense in $X$.
\item \label{item:nuLim} For every $x \in X$ with $\|x \|=1$, there is an $x^* \in J(x)$ such that 
$$\lim_{\mylambda \rightarrow \infty} \inner{f(\mylambda x) - \mylambda x, x^*} = - \infty.$$
\end{enumerate}
For these conditions, the following implications are true. 
\begin{itemize}
\item Condition \ref{item:SaBounded} implies \ref{item:denseDisp}, \ref{item:fu} implies \ref{item:surjDisp}, and \ref{item:denseDisp} implies \ref{item:nuLim}.  
\item If $X$ has the fixed point property, then \ref{item:SaBounded} implies \ref{item:fu}.
\item If $X$ is uniformly smooth, then conditions \ref{item:denseDisp} and \ref{item:nuLim} are equivalent. 
\item If $X$ is a Hilbert space, then conditions \ref{item:fu} and \ref{item:surjDisp} are equivalent. 
\item If $X$ is finite dimensional with a smooth or polyhedral norm, then all five conditions are equivalent.
\end{itemize}
\end{theorem}

\begin{proof}
\ref{item:SaBounded} $\Rightarrow$ \ref{item:denseDisp}. Suppose by way of contradiction that condition \ref{item:denseDisp} is false.  Then there is a $u \in X\backslash \overline{\range}(f-\id)$ and $\rho = \inf_{x \in X} \|f(x) - u - x \| > 0$. By Theorem \ref{thm:GaVi}, there is a horofunction $h$ on $X$ such that $h(f(x)-u) \le h(x) - \rho$ for all $x \in X$.  Let the sequence $x_k$ be defined recursively by $x_0 = 0$ and $x_{k+1} = f(x_k)-u$.  Then $h(x_k) \le - \rho k$ for all $k \ge 0$.  Since $h$ is 1-Lipschitz, this implies that $\|x_k \| \rightarrow \infty$.  At the same time, for all $k$ we have
\begin{align*}
\|f(x_k)-x_k\| &\le \|f(x_k)-u - x_k \|+\|u\| & \text{(triangle inequality)}\\
&\le \|f(0)-u\|+\|u\|. & \text{(since } f-u \text{ is nonexpansive)} 
\end{align*}
So $x_k$ is an unbounded sequence contained in the set $\Fix_\delta(f)$ with $\delta = \|f(0)-u\|+\|u\|$ which contradicts \ref{item:SaBounded}. 

\ref{item:denseDisp} $\Rightarrow$ \ref{item:nuLim}. Suppose that condition \ref{item:nuLim} fails, that is, there is a constant $c \in \R$ and $x \in X$ with $\|x\|=1$ such that $\lim_{\mylambda \rightarrow \infty} \inner{f(\mylambda x) - \mylambda x, x^*} \ge c$ for all $x^* \in J(x)$. By Lemma \ref{lem:avoidedCone}, the range of $f-\id$ does not contain any elements in the nonempty open set $\{y \in X : \max_{y^* \in J(x)} \inner{y,y^*} < c \}$.  This contradicts \ref{item:denseDisp}. 

\ref{item:SaBounded} $\Rightarrow$ \ref{item:fu}. We assume here that $X$ has the fixed point property. Let $x, u \in X$.  Then $\|f(x)-x\| \le \|f(x)+u-x\|+\|u\|$ by the triangle inequality.  Therefore $\Fix_\delta(f+u) \subseteq \Fix_{(\delta+\|u\|)}(f)$ is bounded for every $\delta \ge 0$. By choosing $\delta$ large enough so that $\Fix_\delta(f+u)$ is nonempty and letting $C = X$, we see that Theorem \ref{thm:Penot} implies that $f+u$ has a fixed point. In addition, the fixed point set $\Fix_0(f+u)$ is bounded because it is contained in $\Fix_{\|u\|}(f)$. 


\ref{item:nuLim} $\Rightarrow$ \ref{item:denseDisp}. Suppose now that $X$ is uniformly smooth.
By Lemma \ref{lem:convexRange}, $\overline{\range}(f-\id)$ is convex.  Suppose there exists $u \in X \backslash \overline{\range}(f-\id)$. By the Hahn-Banach theorem, there is a linear functional $x^* \in X^*$ such that $\inner{f(x)-x,x^*} > \inner{u, x^*}$ for all $x \in X$. We may assume that $\|x^*\|_* = 1$. Since any uniformly smooth Banach space is reflexive, there exists $x \in X$ such that $x^* \in J(x)$, and since $X^*$ is uniformly convex, $x^*$ is the only element of $J(x)$. Then the fact that $\inner{f(\mylambda x)-\mylambda x, x^*} > \inner{u, x^*}$ for all $\mylambda > 0$ contradicts \ref{item:nuLim}.     


It is obvious that \ref{item:fu} implies \ref{item:surjDisp}, so we will just prove the converse when $X$ is a Hilbert space.

\ref{item:surjDisp} $\Rightarrow$ \ref{item:fu}. Let $X$ be a Hilbert space. It suffices to prove that the fixed point set $\Fix_0(f)$ is bounded since the same argument will apply to $f+u$ for any $u$.  Choose any $w \in \Fix_0(f)$ and $y \in X$.  Since $f$ has surjective displacement, there is an $x \in X$ such that $f(x) = x+y$.  Then $\|x+y - w\|^2 = \|f(x)-f(w)\|^2 \le \|x-w\|^2$ by nonexpansiveness.  Using the inner-product to expand the norms in this inequality, and then canceling terms, we find that $\inner{w,y} \ge \inner{x,y} + \tfrac{1}{2}\|y\|^2$.  We can also find a corresponding upper bound $\inner{w,y} \le \inner{x,y} - \tfrac{1}{2}\|y\|^2$ by finding an $x \in X$ such that $f(x) = x-y$. So there is a constant $b_y > 0$, depending only on $y$, such that $|\inner{w,y}| \le b_y$ for all $w \in \Fix_0(f)$.  Thus $\Fix_0(f)$ is weakly bounded which means it must also be strongly bounded \cite[Corollary 2.4]{Brezis}. 

Now suppose that $X$ is finite dimensional and is either smooth or has a polyhedral norm.  By Theorem \ref{thm:char}, conditions \ref{item:SaBounded}, \ref{item:fu}, and \ref{item:surjDisp} are equivalent, and the above results tell us that \ref{item:SaBounded} implies \ref{item:denseDisp} and \ref{item:denseDisp} implies \ref{item:nuLim}.  Therefore, it suffices to prove that \ref{item:nuLim} implies \ref{item:SaBounded}. 

\ref{item:nuLim} $\Rightarrow$ \ref{item:SaBounded}. Suppose that condition \ref{item:SaBounded} is not true.  By Theorem \ref{thm:char}, there is a horofunction $h$ on $X$ and a constant $c \in \R$ such that $h(f(x)) \le h(x) + c$ for all $x \in X$. Since we are assuming that either $X$ is smooth or $X$ has a polyhedral norm, Theorem \ref{thm:Cormac} implies that $h$ is given by \eqref{eq:Cormac} for some proper face $F^*$ of $B_1^*$ and some $u \in X$ such that $\sup_{\nu \in F^*} \inner{u,\nu} = 0$. Also, when $X$ is smooth or has a polyhedral norm, all proper faces of $B^*_1$ are exposed faces. Thus, there is an $x \in X$ with $\|x \| =1$ such that $F^* = J(x)$.   
Now, for any $x^* \in F^*$ and $t > 0$,  
\begin{align*}
\inner{ u - f(\mylambda x), x^*} &\le \sup_{\nu \in F^*} \inner{u-f(\mylambda x),\nu} \\
&= h(f(\mylambda x)) & (\text{by } \eqref{eq:Cormac})\\
&\le h(\mylambda x) + c \\ 
&= \sup_{ \nu \in F^*} \inner{u-\mylambda x, \nu} +c & (\text{by } \eqref{eq:Cormac} \text{ again})\\
&= -\mylambda + c. & (\text{since } J(x) = F^*)
\end{align*}
By rearranging terms, we have $\inner{u, x^*} - c \le \inner{f(t x),x^*} - t = \inner{f(\mylambda x) - \mylambda x, x^*} $. Since this is true for all $x^* \in J(x)$ and $t > 0$, we have contradicted condition \ref{item:nuLim}. 
\end{proof}

\begin{remark}
The implication \ref{item:SaBounded} $\Rightarrow$ \ref{item:denseDisp} appears to already be known. In fact, the following is \cite[Exercise 13.6(a)]{Deimling}: If $F$ is an $m$-accretive operator (also known as a hyperaccretive operator) on a Banach space $X$ such that $\{x \in X : \|F(x)\| \le \alpha\}$ is bounded for all $\alpha$, then $F$ has dense range. For any nonexpansive map $f: X \rightarrow X$, the map $\id  -f$ is $m$-accretive, so this exercise implies that \ref{item:SaBounded} $\Rightarrow$ \ref{item:denseDisp}. If $X$ has the fixed point property, then \ref{item:SaBounded} $\Rightarrow$ \ref{item:surjDisp} is also known. In that case, if $F$ is an $m$-accretive operator such that $\{x \in X: \|F(x)\| \le \alpha\}$ is bounded for all $\alpha$, then \cite[Theorem 6]{ReTo80} implies that $F$ is surjective.
\end{remark}

\begin{remark}
If $X$ is a Hilbert space and $f:X \rightarrow X$ is nonexpansive, then the multivalued map $u \mapsto F_0(f+u)$ is maximally monotone.  Since monotone operators on a Hilbert space are locally bounded on the interior of their domains \cite[Theorem 23.2]{Deimling}, this gives an alternative proof of the equivalence of conditions \ref{item:fu} and \ref{item:surjDisp}.  This also implies that if $f$ has surjective displacement, then the set of $u$ such that $F_0(f+u)$ is single-valued is a dense $G_\delta$ subset of $X$ \cite[Theorem 23.6]{Deimling}. 
\end{remark}

In infinite dimensions, it is easy to construct nonexpansive maps with dense displacement, but not surjective displacement. 

\begin{example} \label{ex:rightshift}
Let $T: \ell_2(\N) \rightarrow \ell_2(\N)$ be the right shift operator 
$$T(x_1,x_2, x_3, \ldots) = (0, x_1, x_2, \ldots).$$ 
If $e_1 = (1,0,0,\ldots)$, then the map $T+e_1$ is an isometry with no fixed points in $\ell_2(\N)$. Thus $T$ does not have surjective displacement. However, $\|T(x) - x\|^2 = 2\|x\|^2 - 2 \inner{T(x),x} > 0$ for all nonzero $x \in \ell_2(\N)$. Therefore, $\inner{T(\mylambda x) - \mylambda x, x} = \mylambda \inner{T(x)-x,x} \rightarrow - \infty$ as $\mylambda \rightarrow \infty$ for all nonzero $x$.  This shows that $T$ has dense displacement by Theorem \ref{thm:abcd}.
\end{example}

Condition \ref{item:SaBounded} is sufficient for a nonexpansive map to have surjective displacement, but the following example shows that it is not necessary in infinite dimensions.  

\begin{example}
Let $X = \ell_2(\N)$, let $\{ e_k \}_{k \in \N}$ denote the standard basis for $\ell_2(\N)$, and let $f: X \rightarrow X$ be the map defined entrywise by
$$f_j(x) = \begin{cases} \tfrac{j}{j+1}(x_j+1) & \text{if } x_j< -1 \\ 0 & \text{if } |x_j|\le 1 \\ \tfrac{j}{j+1}(x_j-1) & \text{if } x_j > 1.\end{cases}$$
Then $f$ is nonexpansive and $f-\id$ is onto.  After all, for any $u \in X$, $|u_j| \le 1$ for all but finitely many $j \in \N$.  Therefore, we can choose $x \in X$ as follows:
$$x_j = \begin{cases}
-(j+1)u_j + j& \text{if } u_j >  1 \\ -u_j & \text{if } |u_j|\le 1 \\ -(j+1)u_j -j& \text{if } u_j < -1,\end{cases}$$
so that $f(x)-x = u$. On the other hand, observe that the approximate fixed point set $F_2(f)$ contains $(k+2) e_k$ for all $k \in \N$. Therefore $f$ has surjective displacement even though the approximate fixed point set $F_2(f)$ is not bounded.
\end{example}

For linear nonexpansive maps, condition \ref{item:surjDisp} implies condition \ref{item:SaBounded} from Theorem \ref{thm:abcd}. 

\begin{lemma}
Suppose $X$ is a real Banach space, $T: X \rightarrow X$ is linear and nonexpansive, and $T$ has surjective displacement. Then all approximate fixed point sets $F_\delta(T)$ are bounded.  
\end{lemma}

\begin{proof}
For any nonzero $x \in X$, 
$$\lim_{t \rightarrow \infty} \inner{T(tx)-tx,x^*} = \lim_{t \rightarrow \infty} t\inner{T(x)-x,x^*} = -\infty$$
for some $x^* \in J(x)$ by Theorem \ref{thm:abcd}.   Therefore $T(x)-x \ne 0$ and so the kernel of $T-\id$ is $\{0\}$.  This means that $T-\id$ is one-to-one and onto, so it is invertible with a bounded inverse.  Then $F_\delta(T) = (T-\id)^{-1} B_\delta$, so $F_\delta(T)$ is bounded for all $\delta \ge 0$.  
\end{proof}

An interesting corollary of Theorem \ref{thm:abcd} is the following. 

\begin{corollary}
Let $X$ be a real Banach space and suppose that $f$ and $g$ are both nonexpansive maps on $X$. If $X$ is uniformly smooth and $f$ has dense displacement, then $cf+(1-c)g$ is a nonexpansive map with dense displacement for all $0 < c < 1$.  If $f$ has surjective displacement and $X$ is finite dimensional with either a smooth or polyhedral norm, then $c f+ (1-c) g$ is a nonexpansive map with surjective displacement for all $0 < c < 1$.  
\end{corollary}

\begin{proof}
It is easy to verify that $c f+(1-c)g$ is nonexpansive for all $0 < c<1$.  By Lemma \ref{lem:monotone}, 
$$\lim_{\mylambda \rightarrow \infty} \inner{c f (\mylambda x) + (1-c) g(\mylambda x) - \mylambda x, x^*}$$ 
exists and is contained in $[-\infty,\infty)$ for all $x \in X$ with $\|x \| = 1$ and $x^* \in J(x)$. Suppose $f$ has dense displacement. Then $f$ satisfies condition \ref{item:nuLim} of Theorem \ref{thm:abcd} and so must $cf+(1-c)g$ for all $0 < c <1$.  Then if $X$ is uniformly smooth, it follows that from Theorem \ref{thm:abcd} that $cf+(1-c)g$ has dense displacement.  If $X$ is finite dimensional and smooth or $X$ has a polyhedral norm, then Theorem \ref{thm:abcd} implies that $cf+(1-c)g$ has surjective displacement.
\end{proof}

Some of the implications from Theorem \ref{thm:abcd} may not be best possible. It is not clear whether condition \ref{item:SaBounded} implies surjective displacement on Banach spaces without the fixed point property or whether condition \ref{item:nuLim} implies surjective displacement on every finite dimensional real Banach space. We also don't know whether conditions \ref{item:fu} and \ref{item:surjDisp} are equivalent on Banach spaces that are not Hilbert spaces.   

\section{Polyhedral norm nonexpansive maps} \label{sec:polyhedral}

The following theorem gives a computational procedure involving a finite number of limits that determine whether or not a nonexpansive map on a space with a polyhedral norm has surjective displacement.

\begin{theorem} \label{thm:faces}
Let $X$ be a finite dimensional real Banach space with a polyhedral norm and suppose that $f: X \rightarrow X$ is nonexpansive.  For every face $F$ of the unit ball $B_1$, choose an $x_F$ in the relative interior of $F$ and an $x_F^*$ in the relative interior of the corresponding dual face $F^*$ of $B_1^*$.  Then $f-\id$ maps onto $X$ if and only if 
\begin{equation} \label{eq:facecondition}
lim_{\mylambda \rightarrow \infty} \inner{f(\mylambda x_F) - \mylambda x_F, x_F^* } = -\infty
\end{equation}
for every proper face $F$.  
\end{theorem}

Of course, Theorem \ref{thm:faces} is only useful as a test for the existence of fixed points if it is possible to calculate the limits in \eqref{eq:facecondition}.  However, given an explicit expression for $f$, this is frequently the case.  Before proving the theorem, we make the following general observation.  

\begin{lemma} \label{lem:limitLink}
Let $X$ be a real Banach space, let $x, y \in X$ with $\|x\| = \|y\| = 1$, and suppose that $J(x) \subseteq J(y)$.  Then for any $x^* \in J(x)$, 
$$\lim_{\mylambda \rightarrow \infty} \inner{f(\mylambda x) -\mylambda x, x^*} \le \lim_{\mylambda \rightarrow \infty} \inner{f(\mylambda y) -\mylambda y, x^*}.$$
In particular, if $J(x) = J(y)$, then the two limits are equal.
\end{lemma}
\begin{proof}
By Lemma \ref{lem:subgradient}, for any $\mymu>0$ there is a $z^* \in J(x)$ such that 
\begin{align*}
0 &= \lim_{\mylambda \rightarrow \infty} \|\mylambda x - \mymu y\| - \inner{\mylambda x - \mymu y, z^*}  \\
&= \lim_{\mylambda \rightarrow \infty} \|\mylambda x - \mymu y\| - \mylambda + \mymu & \text{(since } z^* \in J(x) \subseteq J(y)) \\
&\ge \limsup_{\mylambda \rightarrow \infty} \| f(\mylambda x) - f(\mymu y) \| - \mylambda + \mymu & (\text{nonexpansiveness}) \\
&\ge \limsup_{\mylambda \rightarrow \infty} \inner{f(\mylambda x) - f(\mymu y), x^*} - \mylambda + \mymu & (\text{since } \|x^*\|_*=1)\\
&= \lim_{\mylambda \rightarrow \infty} \langle f(\mylambda x) - \mylambda x, x^* \rangle - \inner{f(\mymu y) -\mymu y, x^*}.
\end{align*}
We have shown that $\inner{f(\mymu y) -\mymu y, x^*} \ge \lim_{\mylambda \rightarrow \infty} \langle f(\mylambda x) - \mylambda x, x^* \rangle $ for all $s > 0$. The conclusion follows by letting $s \rightarrow \infty$. 
\end{proof}

\begin{proof}[Proof of Theorem \ref{thm:faces}]
Choose any proper face $F$ of $B_1$. 
If $f - \id$ maps onto $X$, then Theorem \ref{thm:abcd} implies that 
$\lim_{\mylambda \rightarrow \infty} \inner{f(\mylambda x_F) - \mylambda x_F, x^*} = -\infty$
for at least one $x^* \in F^*$. Any $x_F^*$ in the relative interior of $F^*$ can be expressed as a non-trivial convex combination of $x^*$ with some other $y^* \in F^*$, that is, $x_F^* = c x^* + (1-c) y^*$ where $0 < c < 1$.  Since the $\lim_{\mylambda \rightarrow \infty} \inner{f(\mylambda x_F) - \mylambda x_F, y^*}<\infty$ by Lemma \ref{lem:monotone}, it follows that $\lim_{\mylambda \rightarrow \infty} \inner{f(\mylambda x_F) - \mylambda x_F, x_F^*} = - \infty$. This proves that the condition given by \eqref{eq:facecondition} is necessary in order for $f - \id$ to map onto $X$.  

Now, suppose that \eqref{eq:facecondition} holds for one $x_F \in \ri F$. For any other $y_F \in \ri F$, $J(x_F) = J(y_F)$. Then by Lemma \ref{lem:limitLink}
$$\lim_{\mylambda \rightarrow \infty} \inner{f(\mylambda y_F) - \mylambda y_F, x_F^*} =\lim_{t \rightarrow \infty} \inner{ f(t x_F) - t x_F, x_F^*} = - \infty$$ 
for all $y_F \in \ri F$.  Therefore confirming \eqref{eq:facecondition} for one representative $x_F$ in the relative interior of each proper face $F$ is sufficient to confirm it for all $y_F \in \ri F$.  Since every $x \in X$ with $\|x \|=1$ is contained in the relative interior of some face $F$, verifying \eqref{eq:facecondition} for every face $F$ will establish condition \ref{item:nuLim} in Theorem \ref{thm:abcd}. That is sufficient to show that $f$ has surjective displacement.
\end{proof}

One might wonder if checking \eqref{eq:facecondition} for a subset of the faces of $B_1$ might be sufficient to prove that \eqref{eq:facecondition} holds for all proper faces?  For the supremum norm on $\R^n$, this is not the case.  The following example shows that for any proper face $F$ of the unit ball in $(\R^n, \|\cdot \|_\infty)$, there is an $\ell_\infty$ nonexpansive map such that \eqref{eq:facecondition} holds for every face except $F$.  

\begin{example} \label{ex:supnorm}
Let $X = \R^n$ with the supremum norm $\|x\|_\infty$. Let $N = \{1,\ldots,n\}$.
For any subset $I \subseteq N$, let $e_I$ be the vector in $\R^n$ defined by 
$$(e_I)_i = \begin{cases}
1 & \text{if } i \in I \\
0 & \text{otherwise.}
\end{cases}$$
The faces of the unit ball in $X$ are the sets 
$$F_{IJ} = \{x \in B_1: x_i = 1 \text{ if } i \in I, x_i = -1 \text{ if } i \in J \}$$ 
where $I$ and $J$ are disjoint subsets of $N$. Note that $F_{IJ}$ is a proper face if at least one of the sets $I, J$ is nonempty. The corresponding dual faces are 
$$F_{IJ}^* = \{ x \in B_1^* : \sum_{i \in I} x_i - \sum_{j \in J} x_j = 1\}.$$
For each face $F_{IJ}$, let $x_{IJ} = e_I - e_J$ and $x_{IJ}^* = \tfrac{1}{|I|+|J|} (e_I - e_J)$. 
Observe that $x_{IJ} \in \ri(F_{IJ})$ and $x_{IJ}^* \in \ri(F_{IJ}^*)$. 

In order to apply Theorem \ref{thm:faces} to a nonexpansive map $f$ on $X$, we must check whether  
$$\lim_{\mylambda \rightarrow \infty} \inner{f(\mylambda x_{IJ}) - \mylambda x_{IJ}, x_{IJ}^*} = -\infty$$
for every disjoint pair $I, J \subseteq N$ with $I \cup J \ne \varnothing$.  

Fix disjoint sets $K,L \subset N$ such that $K \cup L$ is nonempty.
Let $P_{KL}: \R^n \rightarrow \R^n$ be the map 
$$P_{KL}(x)_i = \begin{cases}
\max \{x_i, 0 \} & \text{if } i \in K \cup L \\
0 & \text{otherwise.} 
\end{cases}$$
Let $D_L:\R^n \rightarrow \R^n$ be the transformation
$$D_L(x)_i = \begin{cases}
-x_i & \text{if } i \in L \\
x_i & \text{otherwise.}
\end{cases}$$
Finally, let $\sigma$ be a permutation of $N$ that is a cyclic permutation of the elements of $K \cup L$ and let $S: \R^n \rightarrow \R^n$ be the permutation operator corresponding to $\sigma$, that is $S(x)_i = x_{\sigma(i)}$. 

Each of the maps $P_{KL}$, $D_L$, and $S$ are nonexpansive on $X$, and therefore so is the composition $f = D_L \circ S \circ P_{KL} \circ D_L$. Observe that $\mylambda x_{KL}$ is a fixed point of $f$ for all $\mylambda > 0$.  Therefore $f$ cannot have surjective displacement.  But every one of the limits \eqref{eq:facecondition} holds, except for the one corresponding to the face $F_{KL}$. 
\end{example}

When $\|\cdot \|$ is either the $\ell_1$ or the $\ell_\infty$ norm on $\R^n$, the condition of Theorem \ref{thm:faces} involves checking $3^n - 1$ limit inequalities.  
Interestingly, according to Kalai's $3^d$ conjecture \cite{Kalai89,SWZ09}, any centrally symmetric $n$-dimensional polytope should have at least $3^n$ nonempty faces.  Thus the number of limits needed to check the conditions of Theorem \ref{thm:faces} for other polyhedral norms on $\R^n$ should always be at least $3^n-1$.   

Unfortunately, there is no way to show that an $\ell_2$-nonexpansive map on $\R^n$ has surjective displacement using a finite set of limits of the type in Theorem \ref{thm:faces}. 
To see why, let $f$ be an orthogonal projection onto a one-dimensional subspace spanned by $v \in \R^n$ with $\|v\|_2 = 1$. For any $x$ that is not a scalar multiple of $v$ (that is, almost every $x$), we have $|\inner{x,v}| < \|x\|_2$, so $\lim_{\mylambda \rightarrow \infty} \inner{f(\mylambda x) - \mylambda x, x} = \mylambda (|\inner{x,v}|^2 - \|x\|_2^2) \rightarrow -\infty$.  This is despite the fact that $f$ has an unbounded set of fixed points.  
In section \ref{sec:recession}, we will give a computable sufficient condition for confirming surjective displacement when $f$ is a nonexpansive map on a smooth finite dimensional normed space.

\section{Perron-Frobenius theory} \label{sec:PF}

Let $N = \{1, \ldots, n\}$. For any subset $I \subseteq N$, let $e_I$ be the vector in $\R^n$ defined by 
$$(e_I)_i = \begin{cases}
1 & \text{if } i \in I \\
0 & \text{otherwise.}
\end{cases}$$
We use $\ge$ to denote the natural partial ordering on $\R^n$: $x \ge y$ if and only if $x_i \ge y_i$ for all $i \in N$. A map $T: \R^n \rightarrow \R^n$ is \emph{order-preserving} if $T(x) \ge T(y)$ whenever $x \ge y$.  $T$ is \emph{additively homogeneous} if $T(x+\mylambda e_N) = T(x) + \mylambda e_N$ for all $x \in \R^n$ and $\mylambda \in \R$ and $T$ is \emph{additively subhomogeneous} if $T(x+\mylambda e_N) \le T(x) + \mylambda e_N$ for all $x \in \R^n$ and $\mylambda \ge 0$.  

If $T: \R^n \rightarrow \R^n$ is order-preserving, then $T$ is nonexpansive with respect to the $\ell_\infty$-norm if and only if it is additively subhomogeneous \cite[Lemma 2.7.2]{LemmensNussbaum}. Such maps are called \emph{subtopical}. Order-preserving maps on $\R^n$ that are additively homogeneous are called \emph{topical}.

\subsection{Surjective displacement for subtopical maps}

The following necessary and sufficient condition for a subtopical map to have surjective displacement only requires computing $2(2^n-1)$ limits as opposed to the $3^n-1$ limits needed by the condition in Theorem \ref{thm:faces} for $\ell_\infty$-nonexpansive maps that are not order-preserving. 

\begin{theorem} \label{thm:subtopical}
Let $T: \R^n \rightarrow \R^n$ be subtopical. Then $T$ has surjective displacement if and only if for every nonempty $I \subseteq N$, 
\begin{equation} \label{eq:subtopical}
\lim_{\mylambda \rightarrow \infty} \inner{T(\mylambda e_I),e_I} - \mylambda |I| = -\infty \text{ and }  \lim_{\mylambda \rightarrow \infty} \inner{T(-\mylambda e_I),e_I} + \mylambda |I| = +\infty.
\end{equation}
\end{theorem}

\begin{proof}
Since $T$ is subtopical, it is nonexpansive with respect to the supremum norm $\|\cdot \|_\infty$. We use the notation from Example \ref{ex:supnorm} for the faces $F_{IJ}$, dual faces $F_{IJ}^*$ and representative points $x_{IJ} \in \ri F_{IJ}$ and $x_{IJ}^* \in \ri F_{IJ}^*$ of the unit ball $B_1$ in $(\R^n, \|\cdot\|_\infty)$.   
By Theorem \ref{thm:faces}, $T$ has surjective displacement if and only if 
$$\lim_{\mylambda \rightarrow \infty} \inner{T(\mylambda x_{IJ}) - \mylambda x_{IJ}, x_{IJ}^*} = -\infty$$ 
for every disjoint pair $I, J \subseteq N$ with $I \cup J \ne \varnothing$. Since $x_{IJ} = (|I|+|J|) x_{IJ}^*$, the limit above is equivalent to
$$\lim_{\mylambda \rightarrow \infty} \inner{T(\mylambda x_{IJ}), x_{IJ}} - \mylambda (|I| + |J|) = -\infty.$$ 
In order for $T$ to have surjective displacement, it is necessary that both limits in \eqref{eq:subtopical} hold true, as those two limits correspond to the limit condition above for $x_{I \varnothing}$ and $x_{\varnothing I}$. We will now show that the limits in \eqref{eq:subtopical} are also sufficient for $T$ to have surjective displacement.  

Note that $-e_J \le x_{IJ} \le e_I$.  Therefore $T(-\mylambda e_J) \le T(\mylambda x_{IJ}) \le T(\mylambda e_I)$ for all $\mylambda > 0$.  So
\begin{align*}
\inner{T(\mylambda x_{IJ}), x_{IJ}} - \mylambda (|I| + |J|) &= \inner{T(\mylambda x_{IJ}), e_I} - \mylambda |I| -  \inner{T(\mylambda x_{IJ}), e_J} - \mylambda |J| \\
&\le \inner{T(\mylambda e_I), e_I} - \mylambda |I| -  \inner{T(-\mylambda e_J), e_J} - \mylambda |J| \\
&= \left( \inner{T(\mylambda e_I), e_I} - \mylambda |I| \right) - \left( \inner{T(-\mylambda e_J), e_J} + \mylambda |J| \right).
\end{align*}
If \eqref{eq:subtopical} holds for all nonempty $I$, then the inequality above implies that
$$\lim_{\mylambda \rightarrow \infty} \inner{T(\mylambda x_{IJ}), x_{IJ}} - \mylambda (|I| + |J|) = -\infty$$
for all disjoint sets $I, J \subseteq N$ with $I \cup J \ne \varnothing$.  Hence $T$ has surjective displacement by Theorem \ref{thm:faces}. 
\end{proof}

\subsection{Existence of additive eigenvectors for topical maps} 
The results in this subsection are not new, but rather give alternative proofs of some necessary and sufficient conditions from \cite{AkGaHo20} for a topical map $T:\R^n \rightarrow \R^n$ to have the property that $T+u$ has an additive eigenvector for all $u \in \R^n$. The results in \cite{AkGaHo20} were proved using a game theoretic approach involving dominions of players in certain one and two-player games. Here we will show how those results follow from Theorem \ref{thm:faces}.

For $x \in \R^n$, let $\mathsf{t}(x) = \max_{i \in N} x_i$ and $\mathsf{b}(x) = \min_{i \in N} x_i$. If $T: \R^n \rightarrow \R^n$ is topical, then $T$ is nonexpansive with respect to the variation seminorm $\|x \|_\text{var} = \mathsf{t}(x)-\mathsf{b}(x)$ \cite[Section 2.2]{GaGu04}. Note that $\|\cdot \|_\text{var}$ is not a norm on $\R^n$, but it is a norm on the subspace $V_0 = \{ x \in \R^n : \inner{x,e_N} = 0 \}$. The normalized map $f(x) = T(x) - \tfrac{1}{n}\inner{T(x), e_N}e_N$ maps $V_0$ into itself and is nonexpansive with respect to $\|\cdot\|_\text{var}$.  If $f$ has a fixed point $x$ in $V_0$, then $x$ is an \emph{additive eigenvector} of $T$, that is, $T(x) = x + \lambda e_N$ for some $\lambda \in \R$. In fact, $T+u$ has an additive eigenvector if and only if $f+u_0$ with $u_0 = u - \tfrac{1}{n}\inner{u,e_N}e_N$ has a fixed point. Since the variation norm is a polyhedral norm on $V_0$, Theorem \ref{thm:faces} gives a testable necessary and sufficient condition to see if $T + u$ has an additive eigenvector for every $u \in \R^n$. We will see below that the condition ends up being equivalent to the one given in \cite[Theorem 4.2]{AkGaHo20}.  

The proper faces of $B_1$ in $(V_0, \|\cdot\|_\text{var})$ are the sets 
$$F_{IJ} = \{ x \in V_0 \, : \, \|x\|_\text{var} = 1, \, x_i = \mathsf{t}(x) \, \forall i\in I, \, x_j = \mathsf{b}(x) \, \forall j \in J\}$$ 
where $I, J \subset N$ are nonempty and disjoint. 
The dual norm is $\|x \|_\text{var}^* = \tfrac{1}{2} \sum_{i \in N} |x_i|$, and the faces of the dual unit ball are 
$$F_{IJ}^* = \{ x \in V_0 \, : \, \|x\|_\text{var}^* = 1, \textstyle{\sum_{i \in I}} x_i = 1, \textstyle{\sum_{j \in J}} x_j = -1 \}. $$ 
Note that if $x \in F_{IJ}^*$, then $x_i > 0$ for all $i \in I$, $x_i < 0$ for $i \in J$ and $x_i = 0$ for $i \in (I \cup J)^c$.  
Let $x_{IJ}^* = \frac{1}{|I|}e_I - \frac{1}{|J|} e_J$, and let $\delta = \|x_{IJ}^*\|_\text{var} = \left( \frac{1}{|I|} + \frac{1}{|J|} \right).$ By construction, $x_{IJ}^* \in \ri F_{IJ}^*$.
If we let $x_{IJ} = \tfrac{1}{\delta} x_{IJ}^*$, then $x_{IJ} \in \ri F_{IJ}$.  

\begin{lemma} \label{lem:connection}
Let $T: \R^n \rightarrow \R^n$ be topical. Let $f:V_0 \rightarrow V_0$ be the corresponding normalized map $f(x) = T(x)-\tfrac{1}{n} \inner{T(x),e_N}e_N$. Let $I, J$ be disjoint, nonempty subsets of $N$. Then 
\begin{equation} \label{eq:facecondition2}
\lim_{\mylambda \rightarrow \infty} \inner{f(\mylambda x_{IJ}) - \mylambda x_{IJ}, x_{IJ}^*} = -\infty
\end{equation}
if and only if 
\begin{equation} \label{eq:hypergraph}
\lim_{\mylambda \rightarrow \infty} \inner{T(-\mylambda e_{I^c}),e_I } = -\infty \text{ or } \lim_{\mylambda \rightarrow \infty} \inner{T(\mylambda e_{J^c}),e_J} = \infty.
\end{equation}
\end{lemma}

\begin{proof}
Start by observing that 
$$-\tfrac{1}{\delta|J|} e_J \le x_{IJ} \le \tfrac{1}{\delta|I|} e_I.$$
This implies that
$$-e_{I^c} \le x_{IJ} - \tfrac{1}{\delta |I|} e_N \le - \tfrac{1}{\delta |I|}e_{I^c}$$
and
$$\tfrac{1}{\delta |J|} e_{J^c} \le x_{IJ} + \tfrac{1}{\delta |J|} e_N \le e_{J^c}.$$
Also, 
\begin{align*}
\inner{f(\mylambda x_{IJ}) - \mylambda x_{IJ}, x_{IJ}^*} &= \inner{T(\mylambda x_{IJ}),x_{IJ}^*} - t \\
&= \tfrac{1}{|I|}\inner{T(\mylambda x_{IJ}) ,  e_I} - \tfrac{1}{|J|} \inner{T(\mylambda x_{IJ}) , e_J} - t \\
&= \tfrac{1}{|I|} \inner{T(\mylambda x_{IJ} - \tfrac{1}{\delta |I|} e_N), e_I} - \tfrac{1}{|J|} \inner{T(\mylambda x_{IJ}+\tfrac{1}{\delta |J|} e_N), e_J}.
\end{align*}
Combining these facts, we see that 
\begin{equation} \label{eq:lowerConnect}
\inner{f(\mylambda x_{IJ}) - \mylambda x_{IJ}, x_{IJ}^*} \ge \tfrac{1}{|I|} \inner{T(-\mylambda e_{I^c}),e_I} - \tfrac{1}{|J|} \inner{T(\mylambda e_{J^c}),e_J}
\end{equation}
and
\begin{equation} \label{eq:upperConnect}
\inner{f(\mylambda x_{IJ}) - \mylambda x_{IJ}, x_{IJ}^*} \le \tfrac{1}{|I|} \inner{T(-\tfrac{\mylambda}{\delta |I|} e_{I^c}),e_I} - \tfrac{1}{|J|} \inner{T(\tfrac{\mylambda}{\delta |J|} e_{J^c}),e_J}.
\end{equation}
As $T$ is order-preserving, both $\inner{T(-\mylambda e_{I^c}),e_I}$ and $\inner{T(\mylambda e_{J^c}),e_J}$ are monotone functions of $\mylambda$. So both $\lim_{\mylambda \rightarrow \infty} \inner{T(-\mylambda e_{I^c}),e_I}$ and $\lim_{\mylambda \rightarrow \infty} \inner{T(\mylambda e_{J^c}),e_J}$ exist. From this and the last two inequalities, the conclusion follows. 
\end{proof}

The next theorem is equivalent to \cite[Theorem 4.2]{AkGaHo20}. 

\begin{theorem} \label{thm:hypergraph}
Let $T : \R^n \rightarrow \R^n$ be topical.  Then $T+u$ has an additive eigenvector for all $u \in \R^n$ if and only if for every pair of nonempty, disjoint sets $I, J \subset N$ either 
$$\lim_{ \mylambda \rightarrow \infty} \inner{T(-\mylambda e_{I^c}), e_I} = - \infty \text{ or } \lim_{ \mylambda \rightarrow \infty} \inner{T(\mylambda e_{J^c}), e_J} = + \infty.$$
\end{theorem}

\begin{proof}
The functions $T+u$ have an additive eigenvector for every $u \in \R^n$ if and only if the corresponding normalized map $f: V_0 \rightarrow V_0$ defined by $f(x) = T(x) - \tfrac{1}{n}\inner{T(x),e_N} e_N$ has surjective displacement. By Theorem \ref{thm:faces}, this happens if and only if \eqref{eq:facecondition2} holds for every disjoint pair of nonempty subsets $I, J \subset N$. This, in turn, is equivalent to \eqref{eq:hypergraph} holding for each pair $I, J$ by Lemma \ref{lem:connection}.
\end{proof}

The condition of Theorem \ref{thm:hypergraph} only requires computing $2(2^n-1)$ limits, however one also needs to check that for each nonempty, disjoint pair $I, J \subset N$, one of the limits in \eqref{eq:hypergraph} is infinite.  There are almost $3^n$ such pairs.  However, a combinatorial approach using hypergraphs reduces the number of computations necessary to check the conditions above to $O(2^n n^2)$. We'll give a short explanation of the idea here, for more details see \cite[Section 4.2.3]{AkGaHo20}.  

A \emph{directed hypergraph} consists of a set of nodes $N$ and a collection of directed hyperarcs.  A \emph{hyperarc} is an ordered pair $(\mathbf{t},\mathbf{h})$ where $\mathbf{t}, \mathbf{h} \subseteq N$.  Following \cite{AkGaHo20}, we introduce two directed hypergraphs $\Hp$ and $\Hm$ both of which have nodes $N = \{1, \ldots, n\}$.  The hyperarcs of $\Hp$ are the pairs $(J,\{i\})$ with $i \notin J$ such that 
$$\lim_{\mylambda \rightarrow \infty} T_i(\mylambda e_J) = +\infty.$$ 
The hyperarcs of $\Hm$ are $(J,\{i\})$ with $i \notin J$ such that 
$$\lim_{\mylambda \rightarrow \infty} T_i(-\mylambda e_J) = -\infty.$$ 

Let $\mathcal{H}$ be any directed hypergraph with nodes $N$. For $I \subseteq N$, we say that $I$ is \emph{invariant} in $\mathcal{H}$ if there are no hyperarcs from a subset of $I$ to a subset of $I^c$.  Since $T$ is order-preserving, if $I$ is not invariant in $\Hm$ or $\Hp$, then there must be a $j \in I^c$ such that $(I,\{j\})$ is a hyperarc in the hypergraph.  Note that $\lim_{\mylambda \rightarrow \infty} \inner{T(-\mylambda e_{I^c}), e_I} = - \infty$ if and only if $I^c$ is not invariant in $\Hm$ and $\lim_{\mylambda \rightarrow \infty} \inner{T(\mylambda e_{J^c}), e_J} = \infty$ if and only if $J^c$ is not invariant in $\Hp$. 

The \emph{reach} of $I$ in a hypergraph $\mathcal{H}$ is the smallest invariant subset of $N$ that contains $I$. 
Using the notion of reach, we can give an improved version of Theorem \ref{thm:hypergraph}.  This next result is described in the comments preceding \cite[Theorem 4.6]{AkGaHo20}. 

\begin{theorem} \label{thm:hypergraph2}
Let $T : \R^n \rightarrow \R^n$ be topical.  Then $T+u$ has an additive eigenvector for all $u \in \R^n$ if and only if for every nonempty proper subset $J \subset N$ either 
$$\lim_{ \mylambda \rightarrow \infty} \inner{T(\mylambda e_{J^c}), e_J} = + \infty \text{ or } \reach(J,\Hm) = N.$$
\end{theorem}

\begin{proof}
If $\lim_{\mylambda \rightarrow \infty} \inner{ T(\mylambda e_{J^c}), e_J}$ is finite for some nonempty proper subset $J \subset N$ and if $I = N \backslash \reach(J, \Hm)$ is nonempty, then $I^c = \reach(J,\Hm)$ is invariant in $\Hm$ so $\lim_{\mylambda \rightarrow \infty} \inner{T(-\mylambda e_{I^c}), e_I} > -\infty$. Therefore Theorem \ref{thm:hypergraph} implies that $T+u$ does not have an additive eigenvector for all $u \in \R^n$. On the other hand, if $\reach(J, \Hm) = N$, then for any nonempty $I$ disjoint from $J$, we have $J \subseteq I^c$, so $\reach(I^c, \Hm) = N$. This means that $I^c$ is not invariant in $\Hm$, so $\lim_{\mylambda \rightarrow \infty} \inner{T(-\mylambda e_{I^c}),e_I} = -\infty$ for all nonempty $I$ disjoint from $J$. Then Theorem \ref{thm:hypergraph} implies that $T+u$ has an additive eigenvector for all $u \in \R^n$.
\end{proof}

Since the reach of a set in $\Hm$ can be computed relatively quickly (see \cite[Lemma 4.5]{AkGaHo20}), the condition in Theorem \ref{thm:hypergraph2} can be checked in $O(2^n n^2)$ computations including limits as a single computation. See \cite[Theorem 4.6]{AkGaHo20} for details. 

Although we chose to state Theorem \ref{thm:hypergraph2} using the hypergraph $\Hm$, the following dual result is also true.
$T+u$ has additive eigenvector for all $u \in \R^n$ if and only if for every nonempty proper subset $I \subset N$ either 
$$\lim_{ \mylambda \rightarrow \infty} \inner{T(-\mylambda e_{I^c}), e_I} = -\infty \text{ or } \reach(I,\Hp) = N.$$
The proof is essentially the same. It turns out that working with the hypergraph $\Hm$, as in Theorem \ref{thm:hypergraph2}, is more convenient for convex topical maps. 

\subsubsection{Convex topical maps} If $T: \R^n \rightarrow \R^n$ is convex in addition to being topical, then we can determine whether or not $T+u$ has an additive eigenvector for all $u \in \R^n$ much more quickly.  Let $\Gp$ be the directed graph 
with nodes $N$ and an arc from $i$ to $j$ if $\lim_{\mylambda \rightarrow \infty} T_i(\mylambda e_{\{j\}}) = \infty$. Note that the arcs in $\Gp$ point in the opposite direction of the corresponding hyperarcs in $\Hp$.  The graph $\Gp$ was introduced in \cite{GaGu04} and used in \cite{AkGaHo20} to prove the following theorem. Recall that the nodes of a directed graph can be partitioned into their strongly connected components, and a strongly connected component is called a \emph{final class} if there are no arcs from that component to any other.  

\begin{theorem}[{\cite[Corollary 4.4]{AkGaHo20}}] \label{thm:convex}
Let $T: \R^n \rightarrow \R^n$ be convex and topical.  Then $T+u$ has an additive eigenvector for every $u \in \R^n$ if and only if $\Gp$ has a unique final class $C$ and $\reach(C,\Hm) = N$.  
\end{theorem}

Because the strongly connected components of $\Gp$ can be computed in $O(n^2)$ time and $\reach(C,\Hm)$ can also be computed in $O(n^2)$ \cite[Lemma 4.5]{AkGaHo20}, the conditions in Theorem \ref{thm:convex} are much faster to check than the conditions of Theorems \ref{thm:hypergraph}/\ref{thm:hypergraph2} when $n$ is large. See \cite[Corollary 4.7]{AkGaHo20}. The graph $\Gp$ also leads to the following (already known) generalization of \cite[Theorem 2]{GaGu04}.  

\begin{theorem} \label{thm:extra}
Let $T: \R^n \rightarrow \R^n$ be topical. If $\Gp$ is strongly connected, then $T+u$ has an additive eigenvector for every $u \in \R^n$.
\end{theorem}
\color{black}

We will prove both Theorem \ref{thm:convex} and Theorem \ref{thm:extra} using the following lemma which relates the graph $\Gp$ with the hypergraphs $\Hp$ and $\Hm$. 
\begin{lemma} \label{lem:conv}
Let $T: \R^n \rightarrow \R^n$ be topical. If $I \subset N$ is invariant in $\Hp$, then $I^c$ is invariant in $\Gp$.  If $T$ is also convex, then for any $J \subset N$ that is invariant in $\Gp$, $J^c$ is invariant in both $\Hp$ and $\Hm$.  
\end{lemma}

\begin{proof}
If $I \subset N$ is invariant in $\Hp$, then $\lim_{t \rightarrow \infty} \inner{T(te_I),e_{I^c}} < \infty$. Since $T$ is order-preserving, 
$$\lim_{t \rightarrow \infty} T_j(te_{\{i\}}) \le \lim_{t \rightarrow \infty} T_j(te_I) < \infty$$
for all $i \in I$ and $j \in I^c$. Therefore $I^c$ is invariant in $\Gp$.   

Now suppose that $T$ is convex and $J \subset N$ is invariant in $\Gp$. This means that
$\lim_{\mylambda \rightarrow \infty} T_i(\mylambda e_{\{j\}}) < \infty$ 
for all $i \in J$ and $j \in J^c$. By the convexity of $T$, 
$$T_i\left(\tfrac{\mylambda}{|J^c|}  e_{J^c} \right) \le \tfrac{1}{|J^c|} \sum_{j \in J^c} T_i(\mylambda e_{\{j\}})$$
therefore $\lim_{\mylambda \rightarrow \infty} T_i(\mylambda e_{J^c}) < \infty$ for all $i \in J$.  This proves that $J^c$ is invariant in $\Hp$. 

Since $T$ is additively homogeneous, $T_i(-\mylambda e_{J^c}) = T_i(\mylambda e_J) - \mylambda$ for all $i \in N$. By convexity
$$
T_i(0) + \tfrac{1}{2} \mylambda = T_i\left( \tfrac{1}{2} \mylambda e_N \right) \le \tfrac{1}{2} (T_i(\mylambda e_J) + T_i(\mylambda e_{J^c})).
$$ 
Rearranging and doubling the inequality above gives 
$$T_i(-\mylambda e_{J^c}) = T_i(\mylambda e_J) - \mylambda \ge 2T_i(0) - T_i(\mylambda e_{J^c}).$$ 
We have already shown that $\lim_{\mylambda \rightarrow \infty} T_i(\mylambda e_{J^c}) < \infty$ when $i \in J$.  So we have $\lim_{\mylambda \rightarrow \infty} T_i(-\mylambda e_{J^c}) > - \infty$ 
for all $i \in J$. Hence $J^c$ is invariant in $\Hm$.
\end{proof}


\begin{proof}[Proof of Theorem \ref{thm:convex}]
If $\Gp$ has two different final classes $I$ and $J$, then both $I$ and $J$ are invariant in $\Gp$ so by Lemma \ref{lem:conv} $I^c$ is invariant in $\Hm$ and $J^c$ is invariant in $\Hp$. This means that 
$$\lim_{t \rightarrow \infty} \inner{T(te_{J^c}),e_J} < \infty \text{ and } \lim_{t \rightarrow \infty} \inner{T(-t e_{I^c}),e_I} > -\infty.$$ 
So $T+u$ does not have an additive eigenvector for all $u \in \R^n$ by Theorem \ref{thm:hypergraph}. 

Now suppose that $C$ is the unique final class of $\Gp$. Since $C$ is invariant in $\Gp$, $C^c$ is invariant in $\Hp$ by Lemma \ref{lem:conv}.  This means that 
$$\lim_{t \rightarrow \infty} \inner{T(te_{C^c}),e_C} < \infty,$$ 
so by Theorem \ref{thm:hypergraph2} it is necessary that $\reach(C,\Hm) = N$ in order for $T+u$ to have an additive eigenvector for every $u \in \R^n$.

We will now show that $\reach(C,\Hm) = N$ is also sufficient. Let $J$ be any nonempty proper subset of $N$ such that $\lim_{t \rightarrow \infty} \inner{T(t e_{J^c}),e_J} < \infty$, i.e., $J^c$ is invariant in $\Hp$. By Lemma \ref{lem:conv}, $J$ is invariant in $\Gp$. 
Any nonempty set that is invariant in $\Gp$ must contain $C$, so if $\reach(C,\Hm) = N$, then $\reach(J,\Hm)=N$ as well. Thus $T+u$ has an additive eigenvector for all $u \in \R^n$ by Theorem \ref{thm:hypergraph2}.
\end{proof}

\begin{proof}[Proof of Theorem \ref{thm:extra}]
If $T: \R^n \rightarrow \R^n$ is topical and $\Gp$ is strongly connected, then Lemma \ref{lem:conv} implies  that there are no nontrivial invariant subsets of $N$ in the hypergraph $\Hp$.  Thus $\lim_{t \rightarrow \infty} \inner{T(te_{J^c}),e_J} = \infty$ for all nonempty proper subsets $J \subset N$, and therefore $T+u$ has an additive eigenvector for all $u \in \R^n$ by Theorem \ref{thm:hypergraph}. 
\end{proof}

%
%

\section{Uniqueness of fixed points} \label{sec:unique}

In this section, we give necessary and sufficient conditions for the uniqueness of fixed points of polyhedral norm nonexpansive maps.  

\begin{lemma} \label{lem:BigRlittler}
Let $X$ be a finite dimensional real Banach space with a polyhedral norm and suppose that $F$ is a proper face of the unit ball $B_1$.  For any $x \in \ri F$, there is a constant $c < 1$ such that 
$$\|Rx - ry\| = R-r$$ whenever $y \in F$ and $0 \le r \le \frac{1}{2}(1-c)R$. 
\end{lemma}

\begin{proof}
Since $x, y \in F$, it follows that $\inner{x,\nu} = \inner{y, \nu} = 1$ for all $\nu$ in the dual face $F^*$.  Let $c = \max \{\inner{x,\nu} \, : \, \nu \in (\ext B_1) \backslash F^* \}$.  Since $x \in \ri F$, it follows that $c < 1$.  If $0 \le r \le \frac{1}{2} (1-c) R$, then $cR + r \le R-r$.  Observe that 
$\inner{Rx-ry, \nu} = R -r$
if $\nu \in F^*$, and $\inner{Rx-ry, \nu} \le cR + r$ if $\nu \in \ext B_1^*$ and $\nu \notin F^*$.  Therefore $\|Rx-ry\| = \max_{\nu \in \ext B_1^*} \inner{Rx-ry,\nu} = R-r$.  
\end{proof}

\begin{theorem} \label{thm:unique}
Let $X$ be a finite dimensional real Banach space with a polyhedral norm, let $f: X \rightarrow X$ be nonexpansive, and suppose that $f(0) = 0$.  Then the following are equivalent.
\begin{enumerate}
\item $f$ has a nonzero fixed point $u \in X$.
\item For some $R > 0$, there is an $x$ contained in the relative interior of a proper face $F_R$ of $B_R$ such that $f(x) \in F_R$.  
\item There is a proper face $G$ of the unit ball $B_1$ and an $\epsilon > 0$ such that $rG$ is invariant under $f$ for all $0 \le r \le \epsilon$.   
\end{enumerate}
\end{theorem}

\begin{proof}
(1) $\Rightarrow$ (2). If $u \ne 0$ is a fixed point of $f$, let $R = \|u\|$ and let $F_R$ be the face of $B_R$ such that $u \in \ri F$.  Then $f(u) = u \in F_R$.  

(2) $\Rightarrow$ (3). Suppose $x \in \ri F_R$ and $f(x) \in F_R$.  Let $G_R \subseteq F_R$ be the face of $B_R$ such that $f(x) \in \ri G_R$. Let $F = \frac{1}{R}F_R$ and $G = \frac{1}{R} G_R$ so that $F$ and $G$ are faces of the unit ball $B_1$. Let $G^*$ be the dual face corresponding to $G$.  Choose any $v \in G$.  Since $G \subseteq F$, we have $\|x - rv\| = R - r$ for all $r > 0$ sufficiently small by Lemma \ref{lem:BigRlittler}.  Then $\|f(x)-f(rv)\| \le R - r$ by nonexpansiveness.  Observe that $\inner{f(x),\nu} = \|f(x)\| = R$ for all $\nu \in G^*$.  At the same time, $\inner{f(rv),\nu} \le \|f(rv)\| \le r$ for all $\nu \in G^*$.  Therefore, for any $\nu \in G^*$ we have
$$R - r \le \inner{f(x)-f(rv), \nu} \le \|f(x)-f(rv)\| \le \|x - rv \| = R-r.$$
This proves that $\inner{f(rv), \nu} = r$ for all $\nu \in G^*$ and therefore $f(rv) \in rG$.  Since $v \in G$ was arbitrary, this proves that $rG$ is invariant under $f$. 

(3) $\Rightarrow$ (1). The faces of $B_1$ are compact and convex. Therefore, if $rG$ is invariant for some $r > 0$, then it contains a nonzero fixed point of $f$ by the Brouwer fixed point theorem.
\end{proof}

\begin{corollary} \label{cor:unique}
Let $X$ be a finite dimensional real Banach space with a polyhedral norm, let $f: X \rightarrow X$ be nonexpansive, and suppose $u \in X$ is a fixed point of $f$. 
For each face $F$ of the unit ball $B_1$, choose $x_F \in \ri F$ and $x_F^* \in \ri F^*$.  Then $u$ is the unique fixed point of $f$ if and only if 
\begin{equation} \label{eq:unique}
\inner{f(u+\mylambda x_F) -u - \mylambda x_F,x_F^*}  < 0
\end{equation}
for all $\mylambda >0$ and every face $F$ of $B_1$.  
\end{corollary}

\begin{proof}
We may assume without loss of generality that $u = 0$ by replacing $f$ with the map $x \mapsto f(u+x)-u$. Since $f$ is nonexpansive and $f(0) = 0$, it follows that $\|f(\mylambda x_F)\| \le \mylambda$ for all $\mylambda > 0$.  Therefore $\inner{f(\mylambda x_F) - \mylambda x_F,x_F^*} \le 0$.  If \eqref{eq:unique} fails for any face $F$ and $\mylambda > 0$, then $\inner{f(\mylambda x_F),x_F^*} = \mylambda $.  This means that $f(\mylambda x_F) \in \mylambda F$ and so $f$ has a nonzero fixed point by Theorem \ref{thm:unique}. Conversely, if \eqref{eq:unique} holds for every face $F$ and $\mylambda > 0$, then none of the faces $\mylambda F$ are invariant under $f$. Therefore Theorem \ref{thm:unique} implies that $0$ is the unique fixed point of $f$. 
\end{proof}

Note that the functions $\mylambda \mapsto \inner{f(u+\mylambda x_F)-u - \mylambda x_F,x_F^*}$ in \eqref{eq:unique} are monotone decreasing by Lemma \ref{lem:monotone}.  Therefore it suffices to verify that 
$$\lim_{\mylambda \rightarrow 0^+} \operatorname{sign} (\inner{f(u+\mylambda x_F) -u - \mylambda x_F,x_F^*}) = -1$$ 
for each face $F$ in order to confirm that $u$ is the unique fixed point of $f$. 

For subtopical maps, we have the following criterion for a fixed point to be unique. 

\begin{theorem}
Let $T : \R^n \rightarrow \R^n$ be subtopical and suppose that $u$ is a fixed point of $T$.  Then $u$ is the unique fixed point of $T$ if and only if for every nonempty $I \subseteq N$, 
$$
\inner{T(u+\mylambda e_{I})-u, e_I} - \mylambda|I| < 0 \text{ and } \inner{T(u-\mylambda e_{I})-u, e_I} + \mylambda|I| >0
$$
for all $t >0$.
\end{theorem}
\begin{proof}
We can assume without loss of generality that $u = 0$ by replacing $T$ with the map $x \mapsto T(u+x)-u$ which is also subtopical. All subtopical maps are nonexpansive with respect to the $\ell_\infty$ norm on $\R^n$. We use the same notation as in Example \ref{ex:supnorm} for the faces $F_{IJ}$ of the $\ell_\infty$ unit ball $B_1$, as well as representative points $x_{IJ} \in \ri F_{IJ}$ and $x_{IJ}^* \in \ri F_{IJ}^*$.  In the proof of Theorem \ref{thm:subtopical}, we showed that
\begin{align*}
(|I|+|J|)\langle T(\mylambda x_{IJ})&-\mylambda x_{IJ}, x_{IJ}^* \rangle = \inner{T(\mylambda x_{IJ}), x_{IJ}} - \mylambda (|I|+|J|) \\ 
&\le  \left( \inner{T(\mylambda e_I), e_I} - \mylambda |I| \right) - \left( \inner{T(-\mylambda e_J), e_J} + \mylambda |J| \right)
\end{align*}
for all disjoint $I, J \subset N$ with $I \cup J \ne \varnothing$. Therefore, \eqref{eq:unique} holds for every face $F_{IJ}$ (including $F_{\varnothing I}$ and $F_{I \varnothing}$) if and only if 
$$\inner{T(\mylambda e_{I}), e_I} - \mylambda|I| < 0 
\text{ and } \inner{T(-\mylambda e_{I}), e_I} + \mylambda|I| > 0$$
for all nonempty $I \subseteq N$ and $t >0$. 
\end{proof}

We can also use the results of this section to give an alternative proof of \cite[Theorem 4.8]{AkGaHo20}. Combining Corollary \ref{cor:unique} with \eqref{eq:lowerConnect} and \eqref{eq:upperConnect} we have the following. 

\begin{theorem}
Let $T : \R^n \rightarrow \R^n$ be topical.  Suppose that $u$ is an additive eigenvector of $T$.  Then $u$ is the unique additive eigenvector of $T$ up to an additive constant if and only if for every pair of nonempty disjoint sets $I, J \subset N$, and all $t > 0$, either 
$$\inner{T(u-\mylambda e_{I^c})-u, e_I} < 0 \text{ or } \inner{T(u+\mylambda e_{J^c})-u, e_J} > 0.$$
\end{theorem}

As is the case with the condition in Theorem \ref{thm:hypergraph}, determining whether an additive eigenvector is unique can be done more quickly using a hypergraph approach, particularly when $T$ is convex in addition to topical.  See \cite[Subsection 4.3]{AkGaHo20} and \cite{AkGa03}.

\section{Recession maps and semiderivatives} \label{sec:recession}

Let $X$ be a real Banach space and $f: X \rightarrow X$ be nonexpansive. If 
$$f_\infty(x) = \lim_{\mylambda \rightarrow \infty} \mylambda^{-1} f(\mylambda x)$$
exists for all $x \in X$, then we call $f_\infty$ the \emph{recession map} for $f$. Recession maps were introduced in \cite{GaGu04} for topical functions, but the same definition applies more broadly. Recession maps don't always exist, but when one does it can provide a sufficient condition for $f$ to have surjective displacement. The following theorem is based on a similar result for topical maps on $\R^n$ \cite[Theorem 13]{GaGu04}. Related infinite dimensional results with additional compactness assumptions can be found in \cite[Theorem 2.7]{KaKi06} and \cite[Theorem 12]{Penot03}.

\begin{theorem} \label{thm:recession}
Let $X$ be a finite dimensional real Banach space and $f: X \rightarrow X$ be nonexpansive. If the recession map $f_\infty$ exists and 0 is the only fixed point of $f_\infty$, then $f$ has surjective displacement.  
\end{theorem}

\begin{proof}
For each $\mylambda > 1$, let $g_\mylambda(x) = \mylambda^{-1} f(\mylambda x) - x$.  Since $\mylambda^{-1} f(\mylambda x)$ is nonexpansive, each of the functions $g_\mylambda$ is 2-Lipschitz. If $f_\infty$ exists, then the functions $g_\mylambda$ converge pointwise to $f_\infty- \id$, and therefore they converge uniformly on the unit sphere $\partial B_1$. If $0$ is the only fixed point of $f_\infty$, then there must be an $\epsilon > 0$ such that  
$$\inf_{x \in \partial B_1} \|f_\infty(x) - x \| = \epsilon.$$  
Then for $\mylambda$ sufficiently large, 
$$\inf_{x \in \partial B_1} \|\mylambda^{-1}f(\mylambda x) - x \| \ge \tfrac{1}{2} \epsilon,$$
or equivalently,
$$\|f(\mylambda x) - \mylambda x \| \ge \tfrac{1}{2} \epsilon \mylambda$$
for all $x \in \partial B_1$. 
This means that the approximate fixed point sets $\Fix_\delta(f)$ are bounded for all $\delta > 0$, which proves that $f-\id$ is onto by Theorem \ref{thm:char}.
\end{proof}

The condition in Theorem \ref{thm:recession} is sufficient, but not necessary for $f$ to have surjective displacement. For example, the $(\R, |\cdot|)$-nonexpansive function
$$f(x) = \begin{cases} 
0 & \text{if } x \le 1 \\
x-\sqrt{x} & \text{ if } x > 1
\end{cases}$$
has surjective displacement even though $f_\infty(x) = x$ for all $x > 0$. 

Theorem \ref{thm:recession} fails in infinite dimensions.  For example, the right shift operator on $\ell_2(\N)$ from Example \ref{ex:rightshift} is linear, so it is its own recession map.  Clearly, 0 is the only fixed point of the right shift operator, and yet the right shift operator only has dense, not surjective, displacement on $\ell_2(\N)$.  However, when $X$ is a uniformly smooth and strictly convex Banach space, we can prove the following. 

\begin{theorem}
Let $X$ be a real Banach space that is uniformly smooth and strictly convex.  Suppose that $f: X \rightarrow X$ is nonexpansive and the recession map $f_\infty$ exists for $f$.  If 0 is the unique fixed point of $f_\infty$, then $f$ has dense displacement.
\end{theorem}

\begin{proof}
Choose any $x \in X$ with $\|x \|=1$.  Since $X$ is uniformly smooth, $J(x)$ contains a single element, $x^*$.  Since $X$ is strictly convex, $x$ is the only element in $B_1$ with $\inner{x,x^*} = 1$.  Therefore, $\inner{f_\infty(x),x^*} < \inner{x,x^*}$. This means that 
$$\lim_{\mylambda \rightarrow \infty} \mylambda^{-1} \inner{f(\mylambda x) - \mylambda x, x^*} < 0$$
so
$$\lim_{\mylambda \rightarrow \infty} \inner{f(\mylambda x) - \mylambda x, x^*} = -\infty,$$
which proves that $f$ has dense displacement by Theorem \ref{thm:abcd}.  
\end{proof}

Another simple observation about recession maps is that they are composable.  

\begin{lemma}
If $X$ is a real Banach space and $f, g:X \rightarrow X$ are nonexpansive maps that have recession maps $f_\infty$ and $g_\infty$ respectively, then the recession map for $f \circ g$ is $f_\infty \circ g_\infty$.  
\end{lemma}

\begin{proof}
Let $f_\mylambda(x) = \mylambda^{-1} f(\mylambda x)$ and $g_\mylambda(x) = \mylambda^{-1}g(\mylambda x)$. Note that $f_\mylambda$ is nonexpansive for all $\mylambda > 0$. Therefore 
\begin{align*}
\|\mylambda^{-1} f(g(\mylambda x)) &- f_\infty(g_\infty(x))\| = \|f_\mylambda (g_\mylambda (x)) - f_\infty(g_\infty(x))\| \\ 
&\le \|f_\mylambda (g_\mylambda(x)) - f_\mylambda(g_\infty(x))\| + \| f_\mylambda(g_\infty(x)) - f_\infty(g_\infty(x))\| \\
&\le \|g_\mylambda(x) - g_\infty(x)\| + \| f_\mylambda(g_\infty(x)) - f_\infty(g_\infty(x))\|.
\end{align*}
Since both $\|g_\mylambda(x) - g_\infty(x)\|$ and $\| f_\mylambda(g_\infty(x)) - f_\infty(g_\infty(x))\|$ approach $0$ as $\mylambda \rightarrow \infty$, we conclude that $f_\infty \circ g_\infty$ is the recession map for $f \circ g$. 
\end{proof}
%

Just like recession maps provide a sufficient condition for surjective displacement, the semiderivative of a nonexpansive map can give a sufficient condition to check uniqueness of fixed points.  Nonexpansive maps aren't always differentiable.  When they aren't, it is sometimes possible to define a semiderivative as follows.  Let $X$ be a real Banach space, and $f: X \rightarrow X$ be a function.  For any $u \in X$, the \emph{semiderivative} of $f$ at $u$ is the function $f'_u: X \rightarrow X$ given by 
$$f'_u(x) = \lim_{t \rightarrow 0^+} \frac{f(u+tx)-f(u)}{t},$$
assuming those limits exist for all $x \in X$. Semiderivatives are similar to recession functions. For example, if $0$ is a fixed point of $f$, then the semiderivative $f'_0$ and the recession function $f_\infty$ are both limits of the expression $t^{-1} f(tx)$, one as $t \rightarrow 0^+$, the other as $t\rightarrow \infty$. When a semiderivative exists at a fixed point, it can be used to check if the fixed point is unique. The following is a special case of \cite[Theorem 6.1]{AkGaNu14}.

\begin{theorem}
Suppose that $X$ is a finite dimensional real Banach space, $f: X \rightarrow X$ is nonexpansive, and $u \in X$ is fixed point of $f$.  If the semiderivative $f'_u$ exists and 0 is the unique fixed point of $f'_u$, then $u$ is the unique fixed point of $f$. 
\end{theorem}

Note that semiderivatives are composable.  In fact, they obey the chain rule. If $f, g$ are nonexpansive maps on a Banach space $X$ and if $g$ is semidifferentiable at $u$ and $f$ is semidifferentiable at $g(u)$, then $(f \circ g)'_u = f_{g(u)}' \circ g_u'$ (see \cite[Lemma 3.4]{AkGaNu14}).

\subsection{Homogeneous nonexpansive maps}
If $f$ is nonexpansive, then both the recession function $f_\infty$ and all semiderivatives $f'_u$ are nonexpansive, assuming they exist.  They are also \emph{homogeneous}, that is, $f_\infty(c x) = c f_\infty(x)$ and $f'_u(c x) = cf'_u(x)$ for all $c > 0$ and $x \in X$.  Homogeneous nonexpansive maps have some interesting properties that are worth mentioning here. As motivation, note that all of the theorems so far in this section involve determining whether 0 is the unique fixed point of a homogeneous nonexpansive map. 

Let $g: X\rightarrow X$ be homogeneous and nonexpansive. By continuity, $0$ must be a fixed point of $g$. When $X$ is finite dimensional, Theorem \ref{thm:recession} implies that 0 is the only fixed point of $g$ if and only if $g$ has surjective displacement since $g$ is its own recession map. On the other hand, if $0$ is not the only fixed point of $g$, then there is a nonzero $u \in X$ such that $\mylambda u$ is a fixed point for all $\mylambda > 0$. Therefore determining whether $0$ is the unique fixed point of a homogeneous nonexpansive map is equivalent to determining whether the fixed point set is bounded. For many finite dimensional real Banach spaces, this question can be answered using a non-deterministic algorithm proposed in \cite{LLN16}. We will give a brief description of the idea here.

We say that a vector $v \in X$ \emph{illuminates} a point $w$ on the boundary of the unit ball if $\|w+\epsilon v\| < 1$ for all sufficiently small $\epsilon > 0$.  Combining \cite[Theorem 3.4 and Corollary 4.5]{LLN16} with the observation above, we have the following result. 

\begin{proposition} \label{prop:illum}
Let $X$ be a finite dimensional real Banach space and $g: X \rightarrow X$ be homogeneous and nonexpansive.  Then 0 is the unique fixed point of $g$ if and only if there are points $w_1, \ldots, w_m \in X$ such that every extreme point of $B_1$ is illuminated by $g(w_i)-w_i$ for some $i$. If $X$ is smooth, then this happens if and only if $0$ is in the interior of the convex hull of $\{g(w_i) - w_i : 1 \le i \le m\}$.
\end{proposition}

This proposition leads immediately to a simple algorithm to check whether $0$ is the unique fixed point of $g$.  Randomly choose $w_i \in X$ and calculate $g(w_i)-w_i$.  Repeat until the set $\{g(w_i)-w_i\}$ is sufficient to illuminate every extreme point of the unit ball.  If this process eventually stops, then $0$ is the unique fixed point. For an $n$-dimensional space with a smooth norm, it only requires $n+1$ vectors to illuminate every extreme point on the unit ball since a set of vectors $\{v_1, \ldots, v_m\}$ illuminates every extreme point exactly when $0 \in \inter \operatorname{conv} \{v_1, \ldots, v_m \}$ \cite[Lemma 4.3]{LLN16}. Therefore there is hope that this illumination algorithm would halt fairly quickly for many homogeneous nonexpansive maps.  By applying the illumination algorithm to the recession map $f_\infty$ or a semiderivative $f'_u$, we can check whether a nonexpansive map has surjective displacement or a unique fixed point. 

For polyhedral norm nonexpansive maps, we already have computable necessary and sufficient conditions for surjective displacement in Theorem \ref{thm:faces} and for uniqueness of fixed points in Corollary \ref{cor:unique}. In some cases however, the illumination algorithm might be quicker to check.  In particular, in $\R^n$ with the $\ell_1$-norm, the unit ball only has $2n$ extreme points, so the illumination algorithm might only need to check on the order of $2n$ limits in order to verify surjective displacement or uniqueness.  For other spaces, such as $\R^n$ with the $\ell_\infty$-norm, this algorithm would require at least $2^n$ limit computations (see \cite[Proposition 4.7]{LLN16}), and therefore would not have much, if any, advantage over the conditions in Theorem \ref{thm:faces} and Corollary \ref{cor:unique}.



On polyhedral normed spaces, homogeneous nonexpansive maps have additional structure that is worth mentioning. Let $X$ be a finite dimensional real Banach space with a polyhedral norm.  Let $g: X \rightarrow X$ be homogeneous and nonexpansive.  Note that $g(B_1) \subseteq B_1$ since $0$ is a fixed point of $g$. For any $x \in B_1$, we will use the notation $F_x$ to indicate the closed face of $B_1$ with $x$ in its relative interior. There is a simple partial order that $B_1$ inherits from its faces. For $x, y\in B_1$, we write $x \preceq y$ if $F_x \subseteq F_y$. We will also say that $x$ and $y$ are \emph{comparable}, and write $x \sim y$, if $x \preceq y$ and $y \preceq x$, that is if $x$ and $y$ are both in the relative interior of the same face of $B_1$. 

\begin{lemma} \label{lem:faceLattice}
Let $X$ be a finite dimensional real Banach space with a polyhedral norm.  Let $g: X \rightarrow X$ be nonexpansive and homogeneous. Then $g$ is order-preserving with respect to $\preceq$. That is, if $x, y \in B_1$ and $F_x \subseteq F_y$, then $F_{g(x)} \subseteq F_{g(y)}$. Consequently, if $x \sim y$, then $g(x) \sim g(y)$. 
\end{lemma}

\begin{proof}
Suppose that $x \preceq y$ for $x, y \in B_1$. If $\|g(y)\| < 1$, then $g(x) \preceq g(y)$ trivially since $F_{g(y)} = B_1$. Let us assume therefore that $\|g(y)\| = 1$. By Lemma \ref{lem:BigRlittler}, for any $\mylambda > 0$ large enough, 
$\|\mylambda y - x\| = \mylambda - 1$. Choose any $\nu \in J(g(y))$. To prove that $F_{g(x)} \subseteq F_{g(y)}$, it suffices to prove that $\inner{g(x),\nu} = 1$. For $\mylambda > 0$ large enough,
\begin{align*}
\mylambda - \inner{g(x),\nu} &= \inner{\mylambda g(y) - g(x), \nu} & (\text{since } \nu \in J(g(y)) )\\
&= \inner{g(\mylambda y) - g(x), \nu} & (\text{homogeneity}) \\
&\le  \|g(\mylambda y) - g(x)\| & (\text{since } \|\nu\|_* = 1)\\
&\le \|\mylambda x - y \| = \mylambda - 1. & (\text{nonexpansiveness})
\end{align*}
From this, we conclude that $\inner{g(x),\nu} \ge 1$, which means that $\inner{g(x),\nu} = 1$ since $\|\nu\|_* =1$.  Therefore $g(x) \in F_{g(y)}$ as claimed.
\end{proof}

An immediate consequence of Lemma \ref{lem:faceLattice} is that homogeneous nonexpansive maps on polyhedral normed spaces induce a well-defined order-preserving mapping on the face lattice of the unit ball as follows: define $\hat{g}(F_x)$ to be $F_{g(x)}$ for any face $F_x$ of $B_1$. Then by Theorem \ref{thm:unique}, $0$ is the unique fixed point of $g$ if and only if $B_1$ is the only fixed point of $\hat{g}$. This suggests another possible route to check for surjective displacement or uniqueness of fixed points. Since the face lattice is a finite partially ordered set, one can  search the face lattice for a fixed point of $\hat{g}$. In some cases it may be possible to use the order-preserving property of $\hat{g}$ to speed up the search.

\subsection*{Acknowledgement} The idea for this paper was suggested by Marianne Akian and St\'{e}phane Gaubert during the 2019 IWOTA conference in Lisbon, Portugal.  The author wishes to thank them both for several discussions during that conference and for their continued encouragement and suggestions since then. Thanks also to the referee and to Roger Nussbaum for their careful reading and helpful comments.  

\bibliography{DW2}
\bibliographystyle{plain}

\end{document}